\documentclass[a4paper,12pt,oneside,reqno]{amsart}
\usepackage{hyperref}
\usepackage[headinclude,DIV13]{typearea}
\areaset{15.1cm}{25.0cm}
\usepackage{amsfonts,amssymb,amsmath,amsthm,bbm}
\usepackage{cancel} 
\usepackage[latin1] {inputenc}
\usepackage{graphicx, psfrag}

\newtheorem{theorem}{Theorem}[section]
\newtheorem{lemma}[theorem]{Lemma}
\newtheorem{proposition}[theorem]{Proposition}
\newtheorem{corollary}[theorem]{Corollary}

\theoremstyle{definition}
\newtheorem{definition}[theorem]{Definition}

\theoremstyle{remark}
\newtheorem*{remark}{Remark}

\def\paragraph#1{\noindent \textbf{#1}}

\numberwithin{equation}{section}

\def\supp{\mathop{\mathrm{supp}}\nolimits}

\def\supp{\mathop{\rm supp}\nolimits}

\def\d{\mathrm{d}}

\def\<{\langle}
\def\>{\rangle}
\def\a{\alpha}
\def\b{\beta}
\def\e{\epsilon}

\def\g{\gamma}

\def\r{\rho}
\def\s{\sigma}
\def\t{\tau}

\def\z{\zeta}
\def\o{\omega}
\def\D{\Delta}

\def\G{\Gamma}

\def\S{\Sigma}

\def\R{{\Bbb R}}  %%
\def\N{{\Bbb N}}  %%
\def\P{{\Bbb P}}  %% carateri piu belle per campi di nombre
\def\E{{\Bbb E}}  %%
\def\T{{\Bbb T}}

\let\cal=\mathcal

\def\EE{{\cal E}}
\def\FF{{\cal F}}

\def\LL{{\cal L}}
\def\MM{{\cal M}}

\def\OO{{\cal O}}

\def\VV{{\cal V}}
\def\UU{{\cal U}}
\def\VV{{\cal V}}

\def\ZZ{{\cal Z}}

 \def \G {{\Gamma}}

 \def \b {{\beta}}
 \def \e {{\varepsilon}}
 \def \s {{\sigma}}
 
 \def \z {{\zeta}}

 \def \D {{\Delta}}
 \def \t {{\tau}}
 
 \def \T {{\Theta}}
 \def \g {{\gamma}}
 
 \def \d {{\delta}}
 \def \a {{\alpha}}
 \def \o {{\omega}}

 \def \r {{\rho}}

 \def \ba {\begin{array}}
 \def \ea {\end{array}}

 \newcommand{\be}{\begin{equation}}
 \newcommand{\ee}{\end{equation}}
\newcommand{\bea}{\begin{eqnarray}}
 \newcommand{\eea}{\end{eqnarray}}
 
\def\TH(#1){\label{#1}}\def\thv(#1){\ref{#1}}
\def\Eq(#1){\label{#1}}\def\eqv(#1){(\ref{#1})}

\def\sfrac#1#2{{\textstyle{#1\over #2}}}

 \def \1{\mathbbm{1}}

\usepackage{cite}
\def\BibTeX{{\rm B\kern-.05em{\sc i\kern-.025em b}\kern-.08em
    T\kern-.1667em\lower.7ex\hbox{E}\kern-.125emX}}

\usepackage{cite}
 \usepackage[tableposition=top]{caption}
\usepackage{txfonts}
\begin{document}
%-------------------------------------------------------------------------------------------------------
\begin{abstract}
In this paper we analyse the genetic evolution of a diploid hermaphroditic population, which is modelled by a three-type nonlinear birth-and-death process with competition and Mendelian reproduction.
In a recent paper, \citen{CMM13} have shown that, on the mutation time-scale, the process converges to the Trait-Substitution Sequence of adaptive dynamics, stepping from one homozygotic state to another with  higher fitness. We prove that,  under the assumption that a dominant allele is also the fittest one, the recessive allele survives for a time of order at least $K^{1/4-\a}$, where $K$ is the size of the population and $\a>0$.
\end{abstract}
%------------------------------------------------------------------------------------------------------

 \title[]
{Survival of a recessive allele in a Mendelian diploid model}

\author[A. Bovier]{Anton Bovier} 
\address{A. Bovier\\Institut f\"ur Angewandte Mathematik\\
Rheinische Friedrich-Wilhelms-Universit\"at\\ Endenicher Allee 60\\ 53115 Bonn, Germany}
\email{bovier@uni-bonn.de }
\author[R. Neukirch]{Rebecca Neukirch}
\address{R. Neukirch\\Institut f\"ur Angewandte Mathematik\\
Rheinische Friedrich-Wilhelms-Universit\"at\\ Endenicher Allee 60\\ 53115 Bonn, Germany}
\email{rebecca.neukirch@iam.uni-bonn.de}

\thanks{We acknowledge financial support from the German Research Foundation (DFG) 
	through the
	\emph{Hausdorff Center for  Mathematics}, the Cluster of Excellence \emph{ImmunoSensation}, and
	the Priority Programme SPP1590 \emph{Probabilistic Structures in Evolution}.
We thank Loren Coquille for help with the numerical simulations and for fruitful discussions.}
\maketitle
\noindent\footnotesize{\emph{Keywords} Adaptive dynamics, population genetics, Mendelian reproduction, diploid population, nonlinear birth-and-death process, genetic variability.}
%--------------------------------------------------------------------------------------------------------
%\vspace{40pt}
%\setcounter{tocdepth}{3}
%\tableofcontents
%\newpage
%--------------------------------------------------------------------------------------------------------
\section{Introduction}
 Mendelian diploid models have been studied for over a century in context of \emph{population genetics} (see e.g. \citen{yule06}, \citen{fisher18}, \citen{wright31}, \citen{haldane24a,haldane24b}). 
 See, e.g., \citen {crowkimura,nagylaki92,ewens04,buerger2000} for text book expositions of population genetics.  
 While population genetics typically deal with models of fixed population size,
 adaptive dynamics, a variant that developed in the 90ies (e.g. \citen{HS90, ML92, MN92}), 
 allows for variable population sizes that are controlled by competition kernels that rule the competitive interaction of populations with different 
 phenotypes resp. geographic locations. Diploid models have been considered in adaptive dynamics as early as 1999 \citen {KG99}. 
 In a recent book by \citen{M12}, the general biological explanation for the adaptive dynamics of Mendelian population was 
developed.
 
 Starting in the mid-90ies, stochastic individual based models 
 were introduced and investigated that allow for a rigorous derivation of many of the predictions of adaptive dynamics on the basis of
 convincing models for populations of interacting individuals that incorporate the canonical genetic  mechanisms of birth, death, mutation, and competition 
 (see, e.g.,  \citen{DL96,C_CEAD,C06,F_MA,C_ME,CM11}). 
 An important and interesting feature of these models is that various scaling limits when the carrying capacity tends to infinity while mutation rates and mutation step-size  tend to zero yield different 
 limit processes on different time-scales. In this way, 
\citen{C06} proof   convergence to the \emph{Trait Substitution Sequence} (TSS) (see, e.g. 
 \citen{DL96, MG96}) and to the \emph{Canonical Equation of Adaptive Dynamics} (CEAD). In \citen{CM11} also the phenomenon of \emph{evolutionary branching} 
 under the assumption of coexistence
 is derived rigorously. 
In a recent paper \citen{B14}  the convergence to the CEAD  is  shown in the simultaneously combined limits of large population, rare mutations and small mutational steps. 

The models considered so far in this context  almost exclusively assume haploid populations with a-sexual reproduction. Exceptions are  the  paper
\citen{CMM13}, where the TSS is derived in a Mendelian diploid model under certain assumptions (that we will discuss below) and more recently some papers by 
\citen{coron, coron2, coron3}.
In the present paper we pick up this line of research and  study a diploid population with Mendelian reproduction similar to the one of 
\citen{CMM13}, but with one notable difference in the assumptions.  
Each individual is characterised by a reproduction and death rate which depend on a phenotypic
 trait (e.g. body size, hair colour, rate of food intake, age of maturity) determined by its genotype, for which there exists two alleles $A$ and $a$ on one single locus. 
We examine the evolution of the trait distribution of the three genotypes $aa, aA$ and $AA$ under  the three basic mechanisms: 
heredity, mutation and selection. Heredity transmits traits to new offsprings and thus ensures the continued existence of the trait distribution. 
Mutation produces variation in the trait values in the population onto which selection is acting. 
Selection is a consequence of competition for resources or area between individual.
\citen{CMM13} have shown that in the limit of large population  and  rare mutations, 
and under the assumption of co-dominance,  the suitably time-rescaled process, converges to the TSS model of adaptive dynamics, 
essentially as shown in \citen{C06} in the haploid case. 
We now reverse the assumption of  \citen{CMM13} that the  alleles $a$ and $A$ are co-dominant and 
assume instead that $A$ is the fittest and 
dominant allele, i.e., the genotypes $aA$ and $AA$ have the same phenotype. We show that this has a dramatic effect on the 
evolution of the population and, in particular, leads to a much prolonged survival of the "unfit" phenotype $aa$ in the 
population.
More precisely, we prove that the mixed type $aA$ decays like $1/t$, in contrast to the exponential decay in 
\citen{CMM13}. This type of behaviour has been observed earlier in the context of population genetic models, see. e.g. \citen{nagylaki92}, Chapter 4. 
The main result of the present  paper is  to show that this fact translated in the \emph{stochastic} model into survival of the less fit phenotype for a time of order almost $K^{1/4}$,
when $K$ is the \emph{carrying capacity} (i.e. the order of the total population size). Let us emphasise that the main difficulty in our analysis is to control the behaviour of the 
stochastic system over a time horizon that diverges like a power of $K$. This precludes in particular the use of functional laws of large numbers, or the like.
 Instead, our proof relies on  the stochastic Euler scheme  developed in \citen{B14}. One could probably give a heuristic derivation of this fact in the context of the
 diffusion approximation in the one locus two allele model of population genetics (see, e.g. \citen{ewens04}), but we are not aware of a reference where this has actually been carried out.

 Sexual reproduction in a diploid population
amounts for the newborn to choose at random one of the two alleles of each parent for its genotype. Hence, unfit alleles 
can survive in individuals with mixed genotype and individuals with a pure genotype are potentially able to reinvade in the population under 
certain circumstances, i.e. a new mutant allele $B$ that appears before the extinction of
the $a$ allele that has strong competition with the $AA$ population but weak competition with the $aa$ population
may lead to a resurgence of the aa population at the expense of the $AA$ population and coexistence of the
types $aa$ and $BB$. This would increase the genetic variability of the population.
 In other words, if we choose the mutation time scale in such a way that there remain enough $a$-alleles in the 
 population when a new mutation occurs and if the new mutant can coexist with the unfit $aa$-individuals, then the 
 $aa$-population can recover.  Numerical simulations show that this can happen but requires subtle tuning of parameters.
 This effect will  be analysed in a forthcoming publication.  Related questions have recently been 
addressed in haploid models by \citen{BiSma15}.

%--------------------------------------------------------------------------------------------------------
\section{Model setup and goals}

\subsection{Introduction of the model}
We consider a Mendelian diploid model introduced by \citen{CMM13}.  
It models a population of a finite number of individuals with sexual reproduction, where
each individual $i$ is characterised by two alleles, $u_1^iu_2^i$, from some  allele space $\mathcal{U}\in\R$.  These two 
alleles define the genotype of individual $i$, which in turn defines its phenotype, $\phi(u_1^iu_2^i)$, through a function 
$\phi:\mathcal{U}^2\rightarrow\R$. We suppress parental effects, thus $\phi(u_1^iu_2^i)=\phi(u_2^iu_1^i)$. 
The individual-based microscopic Mendelian diploid model is a non-linear stochastic birth-and-death process. Each 
individual has a Mendelian reproduction rate with mutation and a natural death rate. Moreover, there is an 
 additional death rate  due to ecological 
competition with the other individuals in the population. 
The following demographic parameters depend all on the phenotype, but we suppress this from the notation. 
Let us define
\vspace{-6pt}
\begin{center}
	\begin{tabular}{ll}
		  $f_{u_1u_2}\in\R_+$ 	& the per capita birth rate (fertility) of an individual with genotype $u_1u_2$.\\
 		 $D_{u_1u_2}\in\R_+$ 	& the per capita natural death rate of an individual with genotype $u_1u_2$.\\
		  $K\in\N$ 				& the parameter which scales the population size.\\
		  $\frac{c_{u_1u_2,v_1v_2}}{K}\in\R_+$ & \parbox[t]{1.0\textwidth}{the competition effect felt by an individual with genotype $u_1u_2$ from \\
		  						an individual with genotype $v_1v_2$.}\\
 		 $\mu_K\in\R_+$ & the mutation probability per birth event. Here it is independent of the \\
		 						&genotype.\\
		$\sigma>0$				&the parameter scaling the mutation amplitude.\\
			 \end{tabular}
 \end{center}
	\begin{center}
	\begin{tabular}{ll}
		$m(u,dh)	$				&\parbox[t]{1.0\textwidth}{mutation law of a mutant allelic trait $u+h\in\mathcal{U}$, born from an\\individual with allelic trait $u$.} 
 	 \end{tabular}
 \end{center}
 \vspace{3pt}
 Scaling the competition function $c$ down by a factor $1/K$ amounts to scaling the  population size to order $K$.
 $K$ is  called the \emph{carrying capacity}. We are interested in asymptotic  results when $K$ is very large.
 We assume rare mutation, i.e. $\mu_K\ll1$. Hence, if a mutation occurs at a birth event, only one allele changes 
 from $u$ to $u+\sigma h$ where $h$ is a random variable with law $m(u,dh)$ and $\sigma\in[0,1]$.

 At any time $t$, there is  a finite number, $N_t$, of individuals, each with  genotype in $\mathcal{U}^2$. 
 We denote by $u_1^1u_2^1,...,u_1^{N_t}u_2^{N_t}$ the genotypes of the population at time $t$. 
 The population, $\nu_t$,  at time $t$ is represented by the rescaled sum of Dirac measures on $\mathcal{U}^2$, 
\begin{align}
\nu_t=\frac{1}{K}\sum_{i=1}^{N_t}\delta_{u_1^iu_2^i}.
\end{align}
$\nu_t$ takes values in
\begin{align}
\mathcal{M}^K=\left\{\frac{1}{K}\sum_{i=1}^{n}
\delta_{u_1^iu_2^i}\Big|n\geq0, u_1^1u_2^1,...,u_1^nu_2^n\in\mathcal{U}^2\right\},
\end{align}
where  $\mathcal{M}$ denotes the set of finite, nonnegative measures on $\mathcal{U}^2$ 
equipped with the vague topology.
Define $\langle\nu,g\rangle$ as the integral of the measurable function $g:\UU^2\rightarrow\R$ with respect to 
the measure 
$\nu\in\MM^K$. Then $\langle\nu_t,\1\rangle=\frac{N_t}{K}$ and for any $u_1u_2\in\mathcal{U}^2$, the positive number 
$\langle\nu_t,\1_{u_1u_2}\rangle$ is called the \emph{density} at time $t$ of the genotype $u_1u_2$. 
The generator of the process is defined as in \citen{CMM13}:
First we define, for the genotypes $u_1u_2,v_1v_2$ and a point measure $\nu$, the Mendelian reproduction operator:
\be
(A_{u_1u_2,v_1v_2}F)(\nu)
=\frac{1}{4}\left[F\left(\nu+\frac{\d_{u_1v_1}}{K}\right)+F\left(\nu+\frac{\d_{u_1v_2}}{K}\right)+F\left(\nu+\frac{\d_{u_2v_1}}{K}\right)+F\left(\nu+\frac{\d_{u_2v_2}}{K}\right)\right]-F(\nu),
\ee
and the Mendelian reproduction-cum-mutation operator:
\begin{align}
(M_{u_1u_2,v_1v_2}F)(\nu)%\nonumber\\
=&\frac{1}{8}\int_{\R}\Bigg[\left(F\left(\nu+\frac{\d_{u_1+hv_1}}{K}\right)+F\left(\nu+\frac{\d_{u_1+hv_2}}{K}\right)\right)
m_\sigma(u_1,h)\nonumber\\
&\hspace{12pt}+\left(F\left(\nu+\frac{\d_{u_2+hv_1}}{K}\right)+F\left(\nu+\frac{\d_{u_2+hv_2}}{K}\right)\right)
m_\sigma(u_2,h)\nonumber\\
&\hspace{12pt}+\left(F\left(\nu+\frac{\d_{u_1v_1+h}}{K}\right)+F\left(\nu+\frac{\d_{u_2v_1+h}}{K}\right)\right)
m_\sigma(v_1,h)\nonumber\\
&\hspace{12pt}+\left(F\left(\nu+\frac{\d_{u_1v_2+h}}{K}\right)+F\left(\nu+\frac{\d_{u_2v_2+h}}{K}\right)\right)
m_\sigma(v_2,h)\Bigg]dh%\nonumber\\
%&\hspace{12pt}
-F(\nu).
\end{align}
The process $(\nu_t)_{t\geq0}$ is then 
 a $\mathcal{M}^K$-valued Markov process with  generator $L^K$, given for any bounded measurable
  function $F:\mathcal{M}^K\rightarrow\R$ by:
\begin{align}
(L^KF)(\nu)
&=\int_{\mathcal U^2}\left(D_{u_1u_2}+\int_{\mathcal U^2}c_{u_1u_2,v_1v_2}\nu(d(v_1v_2))\right)\left(F\left(\nu-\frac{\d_{u_1u_2}}{K}\right)-F(\nu)\right)K\nu(d(u_1u_2))\nonumber\\
&+\int_{\mathcal U^2}(1-\mu_K)f_{u_1u_2}\left(\int_{\mathcal U^2}\frac{f_{v_1v_2}}{\langle\nu,f\rangle}(A_{u_1u_2,v_1v_2}F)(\nu)\nu(d(v_1v_2))\right)K\nu(d(u_1u_2))\nonumber\\
&+\int_{\mathcal U^2}\mu_Kf_{u_1u_2}\left(\int_{\mathcal U^2}\frac{f_{v_1v_2}}{\langle\nu,f\rangle}(M_{u_1u_2,v_1v_2}F)(\nu)\nu(d(v_1v_2))\right)K\nu(d(u_1u_2)).
\end{align}
The first non-linear term describes the competition between individuals. 
The second and last linear terms describe the birth without and with mutation. There, 
$f_{u_1u_2}\frac{f_{v_1v_2}}{K\langle\nu,f\rangle}$ is the reproduction rate of an individual with genotype $u_1u_2$ with 
an individual with genotype $v_1v_2$. Note that we assume random mating with multiplicative fertility (i.e. that birth rate is proportional to the product of the fertilities of the mates).

For all $u_1u_2,v_1v_2\in\mathcal{U}^2$, we make the following Assumptions \textbf{(A)}:
\begin{itemize}
\item[\textbf{(A1)}]The functions $f,D$ and $c$ are measurable and bounded, which means that there exists $\bar f,\bar D,\bar c<\infty$ such that 
\begin{align}
0\leq f_{u_1u_2}\leq\bar f, \quad 0\leq D_{u_1u_2}\leq\bar D \quad \text{and}\quad 0\leq c_{u_1u_2,v_1v_2}\leq\bar c.\vspace{0pt}
\end{align}	
\item[\textbf{(A2)}] $f_{u_1u_2}-D_{u_1u_2}>0$ and there exists $\underline c>0$ such that 
$\underline c\leq c_{u_1u_2,v_1v_2}$.\vspace{6pt}
\item[\textbf{(A3)}] For any $\sigma>0$, there exists a function, $\bar m_\sigma:\R\rightarrow\R_+$, such that 
$\int\bar m_\sigma(h)dh<\infty$ and $m_\sigma(u,h)\leq\bar m_\sigma(h)$ for any $u\in\mathcal{U}$ and $h\in\R$.
\end{itemize}
For fixed $K$, under the Assumptions \textbf{(A1)} + \textbf{(A3)} and assuming that $\E(\langle\nu_0,\1\rangle)<\infty$, \citen{F_MA} have shown existence and uniqueness in law of a process with infinitesimal generator $L^K$. For $K\rightarrow\infty$, under more restrictive assumptions, and assuming the convergence of the initial condition, they prove the convergence in $\mathbb{D}(\R_+,\mathcal{M}^K)$ of the process $\nu^K$ to a deterministic process, which is the solution to a non-linear integro-differential equation. Assumption \textbf{(A2)} ensures that the population does not explode or becomes extinct too fast.
%-----------------------------
\subsection{Goal} 
We start the process with a monomorphic $aa$-population, where one mutation to an $A$-allele has already occurred. 
That 
means, the initial  population consists only of individuals with genotype $aa$ except one individual with genotype $aA$.
The mutation probability for an individual with genotype $u_1u_2$ is given by  $\mu_K$. Hence, the time 
until the next mutation in the whole population is of order $\frac{1}{K\mu_K}$. 
Since the time a mutant population needs to invade a resident population is of order $\ln K$ (see, e.g. \citen{CM11}), we set the mutation rate
 $\frac{1}{K\mu_K}\gg\ln(K)$ in order
  to be able to consider the fate of the mutant and the resident population without the occurrence of a new 
 mutation. In this setting, the allele space $\UU=\{a,A\}$ consist only of two alleles.
  Our results will imply that if the mutation rate is bigger than $\frac1{KK^{1/4-\a}}$, $\a>0$, then a mutation will occur while  the 
 resident $aa$-population is small, but still alive, in contrast to the setting of  \citen{CMM13}, where the 
 $a$-allele dies out by time $\ln K$. 
This different behaviour can be traced to the deterministic system that arises in the large $K$ limit.  
 Figure \ref{fig1} ($A$-allele fittest and dominant) and Figure \ref{fig5} ($a$ and $A$ alleles co-dominant) show 
 a simulation of the deterministic system of the two different models. We see that in the settings of \citen{CMM13} the 
 mixed type $aA$ dies out exponentially fast, whereas in the model where $A$ is the dominant allele, the mixed type 
 decays much slowly. We will show below that this is due to the fact that, 
 under our hypothesis, the stable fixed point of the deterministic system is degenerate, leading to an algebraic rather than 
 exponential approach to the fixed point.
 The main task is to prove that this translates into a survival of the unfit allele in the stochastic model for 
 a time of order $K^\b$. We show that this is indeed the case, with $\b= 1/4-\a$.
This implies that for mutation rates of order $1/K\ln K$, a further mutant will occur in the $AA$ population \emph{before} 
the $aa$ allele is extinct.

%------------------------------
\subsection{Assumptions on the model}
Let $N_{uv}(t)$  be the number of individuals with genotype $uv\in\{aa,aA,AA\}$ in the population at time $t$ and 
set $n_{uv}(t)\equiv \frac{1}{K}N_{uv}(t)$.
\begin{definition}
The \emph{equilibrium size} of a monomorphic $uu$-population, $u\in\{a,A\}$, is the fixed point of  a
1-dim Lotka-Volterra 
equation and is given by 
\begin{align}
\label{equi}
\bar n_u=\frac{f_{uu}-D_{uu}}{c_{uu,uu}}.
\end{align} 
\end{definition}

\begin{definition}
For any $u,v\in\{a,A\}$,
\begin{align}
S_{uv,uu}=f_{uv}-D_{uv}-c_{uv,uu}\bar n_{u}
\end{align}
is called the \emph{invasion fitness} of a mutant $uv$ in a resident $uu$-population, where $\bar n_{u}$ is 
given by \eqv(equi).
\end{definition}

 We assume that the  dominant $A$-allele defines the phenotype of an individual, i.e. 
$AA$ and $Aa$ individuals have the same   phenotype.
In particular,  the fertility and the natural death rates are the 
same for $aA$- and $AA$-individuals. For simplicity,  we assume that the competition rates are the same 
for all the three different genotypes.
To sum up, we make the following Assumptions \textbf{(B)} on the rates: 
\begin{itemize}
\item[\textbf{(B1)}] $f_{aa}= f_{aA}\equiv f_{AA}=:f$,\vspace{3pt}	 
\item[\textbf{(B2)}] $D_{AA}\equiv D_{aA}=:D$\hspace{9pt} \emph{but} \hspace{9pt}$D_{aa}= D+\D$,\vspace{3pt}
\item[\textbf{(B3)}]$c_{u_1u_2,v_1v_2}=: c,$ \hspace{9pt}  $\forall u_1u_2,v_1v_2\in\{aa,aA,AA\}$.
\end{itemize}
\vspace{6pt}
\begin{remark}
We choose constant fertilities and constant competition rates for simplicity.
 What is really needed, it that the fitness of the $aA$ and $AA$ types are equal and higher than that of the $aa$ type.
\end{remark}

Observe that, under Assumptions  \textbf{(B)},
\begin{align}
\label{infit}
S_{aA,aa}=S_{AA,aa}&=f-D-c\bar n_{aa}=f-D-c\frac{f-D-\D}{c}=\D,\nonumber\\
S_{aa,aA}=S_{aa,AA}&=f-D-\D-c\bar n_{AA}=f-D-\D-c\frac{f-D}{c}=-\D.
\end{align}
Therefore, the $aA$-individuals are as fit as the $AA$-individuals and both are fitter than the $aa$-individuals.
In the 
Mendelian diploid model an individual chooses another individual   uniformly at random as partner for reproduction.
For example, if we want to produce an individual with genotype $aa$, there are four possible combinations for the parents: 
$aa\leftrightarrow aa, aa\leftrightarrow aA, aA\leftrightarrow aa$ and $aA\leftrightarrow aA$. The first combination results in 
an $aa$-individual with probability 1, the second and third one with probability $\frac{1}{2}$ and the last one with probability 
$\frac{1}{4}$. Using analogous reasoning for the other possible cases 
yields the following birth rates:
\begin{align}
b_{aa}(N_{aa}(t),N_{aA}(t),N_{AA}(t))
&=\frac{f\left(N_{aa}(t)+\frac{1}{2}N_{aA}(t)\right)^2}{N_{aa}(t)+N_{aA}(t)+N_{AA}(t)},\nonumber\\
b_{aA}(N_{aa}(t),N_{aA}(t),N_{AA}(t))
&=\frac{2f\left(N_{aa}(t)+\frac{1}{2}N_{aA}(t)\right)\left(N_{AA}(t)+\frac{1}{2}N_{aA}(t)\right)}{N_{aa}(t)+N_{aA}(t)+N_{AA}(t)},\nonumber\\
b_{AA}(N_{aa}(t),N_{aA}(t),N_{AA}(t))
&=\frac{f\left(N_{AA}(t)+\frac{1}{2}N_{aA}(t)\right)^2}{N_{aa}(t)+N_{aA}(t)+N_{AA}(t)}.
\end{align}
The death rates are the sum of the natural death and the competition:
\begin{align}
&d_{aa}(N_{aa}(t),N_{aA}(t),N_{AA}(t))=N_{aa}(t)\left(D+\D+c(n_{aa}(t)+n_{aA}(t)+n_{AA}(t))\right),\nonumber\\
&d_{aA}(N_{aa}(t),N_{aA}(t),N_{AA}(t))=N_{aA}(t)\left(D+c(n_{aa}(t)+n_{aA}(t)+n_{AA}(t))\right),\nonumber\\
&d_{AA}(N_{aa}(t),N_{aA}(t),N_{AA}(t))=N_{AA}(t)\left(D+c(n_{aa}(t)+n_{aA}(t)+n_{AA}(t))\right).
\end{align}
In the sequel, the \emph{sum process}, $\S(t)$,  defined by
\begin{align}
\label{sumprocess}
\S(t)=n_{aa}(t)+n_{aA}(t)+n_{AA}(t),
\end{align}
plays an important role. 
A simple calculation shows that the sum process jumps up (resp. down) by rate $b_\Sigma$ (resp. $d_\Sigma$) given by
\be
b_\S(N_{aa}(t),N_{aA}(t),N_{AA}(t))=fK\S(t),\qquad d_\S(N_{aa}(t),N_{aA}(t),N_{AA}(t))=DK\S(t)+\D N_{aa}(t).
\ee

%--------------------------------------------------------------------------------------------------------
\section{Main theorems}
\label{sectionmain}
In the sequel we denote by $\t_n$, $n\geq1$, the ordered sequence of times when a   mutation occurs in the population. 
We assume $\t_0=0$
and make following Assumption \textbf{C} on the mutation rate $\mu_K$:
\begin{align}
\label{mu}
\text{\textbf{(C)}}\quad\quad\quad\ln(K)\ll\frac{1}{K\mu_K}\ll K^{1/4-\a}.
\end{align}
We recall one result from \citen{CMM13} (Proposition D.2) which carry over to our setting:
it is shown that if the resident population $n_{aa}(t)$ is in a $\d$-neighbourhood of 
its equilibrium $\bar n_a$, then $n_{aa}(t)$  stays in this neighbourhood for an exponentially long time, as long as the 
mutant population is smaller than $\d$. The proof of this result is based on large deviation estimates  (see \citen{FW84}). 
\begin{proposition}[Proposition D.2 in \citen{CMM13}]
\label{propD}
Let $\supp(\nu_0^K)=\{aa\}$ and let $\t_1$ denote the first mutation time. For any sufficiently small $\d>0$, if 
$\langle\nu_0^K,\1_{aa}\rangle$ belongs to the $\d/2$-neighbourhood of $\bar n_a$ then the time of exit of 
$\langle\nu_t^K,\1_{aa}\rangle$ from the $\d$-neighbourhood of $\bar n_a$ is bigger than $e^{VK}\land\t_1$, for $V>0$, 
with probability converging to 1.
Moreover, there exists a constant $M$, such that, for any sufficiently small $\d>0$, this remains true, if the death rate of an 
individual with genotype $aa$,
\begin{align}
D+cK\langle\nu_t^K,\1_{aa}\rangle,
\end{align}
is perturbed by an additional random process that is uniformly bounded by $M\d$.
\end{proposition}
We start the population process when $n_{aa}$ is in a 
$\d/2$-neighbourhood of its equilibrium, $\bar n_a$, and  there is one individual with genotype $aA$. 
The first theorem says that there is a positive probability that the mutant population fixates in the resident $aa$-population and the second theorem gives the time for the invasion of the mutant population and a lower bound on the survival time 
of the recessive $a$-allele. Define
\begin{align}
\label{stopmut}
\t_{\d}^{mut}&\equiv\inf\{t\geq0: 2n_{AA}(t)+n_{aA}(t)\geq \delta\}, \\
\t_{0}^{mut}&\equiv\inf\{t\geq0: 2n_{AA}(t)+n_{aA}(t)=0\}.
\end{align}
\begin{theorem}[Proposition D.4 in \citen{CMM13}]
\label{fix}
Let $(z_K)$ be a sequence of integers such that $\frac{z_K}{K}$ converges to $\bar n_{a}$, for $K\rightarrow\infty$. Then
\begin{align}
\lim_{\d\rightarrow0}\lim_{K\rightarrow\infty}\mathbb{P}_{\frac{z_K}{K}\delta_{aa}+\frac{1}{K}\delta_{aA}}(\tau_\d^{mut}<\tau_0^{mut})=\frac{\Delta}{f},
\end{align}
where we recall that $\D$ is the invasion fitness of a mutant $aA$ in a resident $aa$-population (cf. \ref{infit}). 
\end{theorem}

We now state the main results of this paper:
\begin{theorem}
\label{main1}
Consider the model verifying Assumptions \textbf{A} and \textbf{B}. 
Let $\d>\e$, $\a>0$ and
\begin{align}
\t^{hit}_\eta&\equiv\inf\{t\geq\t_\d^{mut}: n_{aA}(t)\leq \eta\}.
\end{align}
Define $\t_{sur}\equiv\t^{hit}_{K^{-1/4+\a}}-\t^{hit}_\e$.
Then, conditional on survival of the mutant, i.e., on  the event $\{\tau_\d^{mut}<\tau_0^{mut}\}$, with probability converging to one as $K\uparrow \infty$, the following statements hold:
\begin{itemize}
\item[(i)]\quad$\t_{\e}^{hit}=\OO(\ln K)$, and
\item[(ii)]\quad$\t_{sur}=\OO(K^{1/4-\a})$. 
\end{itemize}
\end{theorem}
\begin{remark}
As long as there are $aA$-individuals in the population, the smaller $aa$-population does not die out, since the $aA$-
population always gives birth to $aa$-individuals. For smaller values of the power $\frac14-\a$ in (ii), the natural fluctuations 
of the big $AA$-population are too high: the death rate of $n_{aa}(t)$ can be too large due to the competition felt by 
$n_{AA}(t)$ and could induce the death of the $aa$-population and hence also of the $aA$-population.
\end{remark}
The next theorem states that if the mutation rate satisfies Assumption \textbf{C}, then  a 
new mutation to a (possibly fitter) allele, $B$, occurs while  some $a$-alleles are still alive. 
This mutation will happen after the 
invasion of the mutant population and when the mixed type $aA$-population already decreased to a small level again. 
More precisely,
\begin{theorem}
\label{main2}
Assume that Assumption \textbf{C} is satisfied.
Then, with probability converging to one,
\begin{align}
\label{result}
\t_\e^{hit}\land\t_1=\t_\e^{hit} \quad\text{and}\quad \t_1\land\t_{0}^{hit}=\t_1.
\end{align}
\end{theorem}
The interest in this result lies in the fact that  a new mutant allele $B$ that appears before the extinction of the $a$ allele that
has strong competition with the $AA$ population but weak competition with the $aa$ population may lead to a resurgence of the 
$aa$ population at the expense of the $AA$ population and coexistence of the types $aa$ and $BB$. Numerical simulations show that 
this can happen but requires subtle tuning of parameters. 
%----------------------------------------
Since the proofs of  the main theorems (\ref{fix}, \ref{main1}, \ref{main2}) have several parts and are quite technical, we first give an outline of them before we turn to the details (Section \ref{proofmain}).
%------------------------------------------------
\subsection{Outline of the proofs}
%--------------------------------------------------
\subsubsection{Heuristics leading to the main theorems} 
The basis of the main theorems is the observation of the different behaviour of the limiting deterministic system, when $A$ 
is the fittest and dominant one and when the alleles are co-dominant \citen{CMM13}. More precisely, they have dissimilar long-term behaviour (cf. Figure \ref{fig1} 
and \ref{fig5}).\\
By analysing the systems one gets that both have the same fixed points $\mathfrak n_{aa}\equiv(\bar n_a,0,0)$ and 
$\mathfrak n_{AA}\equiv(0,0,\bar n_A)$. The computation of the eigenvalues of the Jacobian matrix at the fixed point 
$\mathfrak n_{aa}$ yields in both models two negatives and one positive eigenvalues. Hence in both systems the fixed point 
$\mathfrak n_{aa}$ is unstable. In contrast, the eigenvalues at the fixed point $\mathfrak n_{AA}$ are all negative in the 
system of \citen{CMM13} whereas in this model there is one zero eigenvalue. This leads to the different long term 
behaviour towards the stable fixed point $\mathfrak n_{AA}$. In \citen{CMM13} model the $aA$-population dies out 
exponentially fast whereas in this model the degenerated eigenvalue corresponds to a decay of $n_{aA}(t)$ like a function 
$f(t)=\frac1t$. The goal is to show that the stochastic system behaves like the deterministic system.

%---------------------------------------------------------------------------------

\begin{figure}[h]
\begin{minipage}[t]{0.49\textwidth}
	\centering
  \includegraphics[width=1\textwidth]{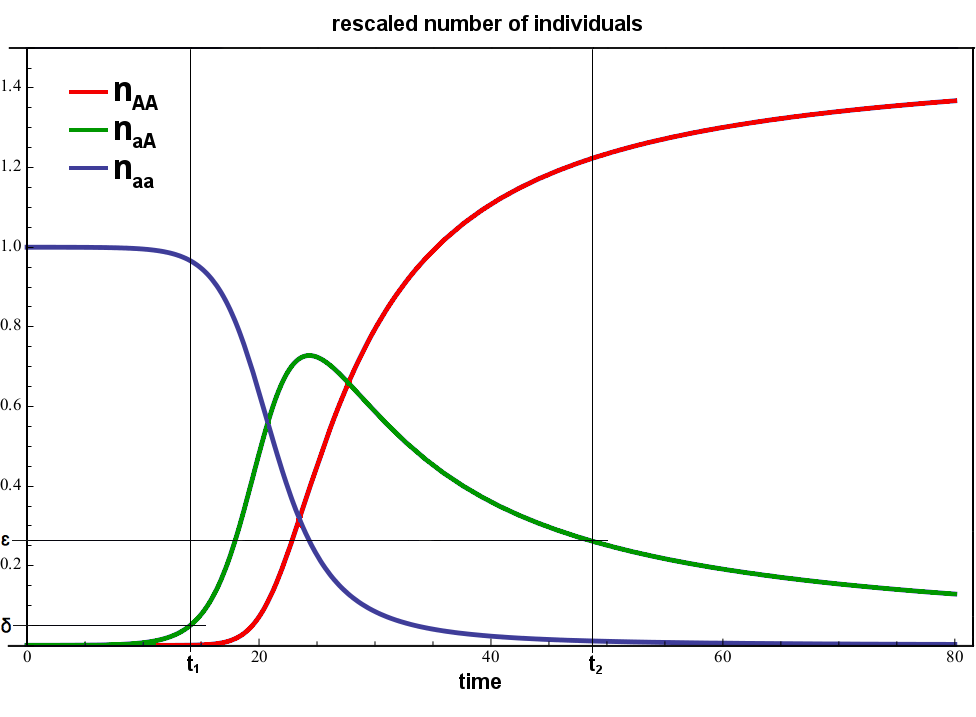}
	\caption{Our Model: $A$ fittest type and dominant}
	\label{fig1}
\end{minipage}
\begin{minipage}[t]{0.5\textwidth}
	\centering
  \includegraphics[width=1\textwidth]{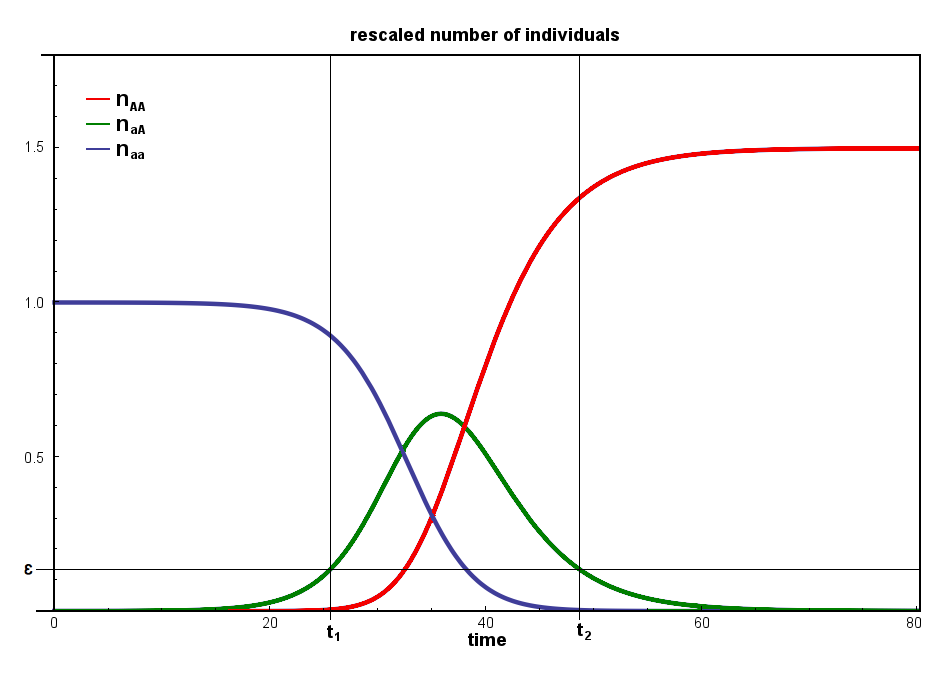}
	\caption{Collet et al. Model: $a$ and $A$ co-dominant}
	\label{fig5}
\end{minipage}
\end{figure}

%-----------------------------------------------------
\subsubsection{Organisation of the proofs}
The main theorems describe the invasion of a mutant in the resident population, and the survival of the recessive allele. 
This invasion can be divided into three phases, in a similar way as in \citen{C06}, \citen{CMM13}, or \citen{B14}
 (cf. Figure \ref{fig2}) (the general idea that an invasion can be divided in these phases is of course much older, see, e.g.  \citen{barton}):
\begin{figure}[t]
	\centering
  \includegraphics[width=0.8\textwidth]{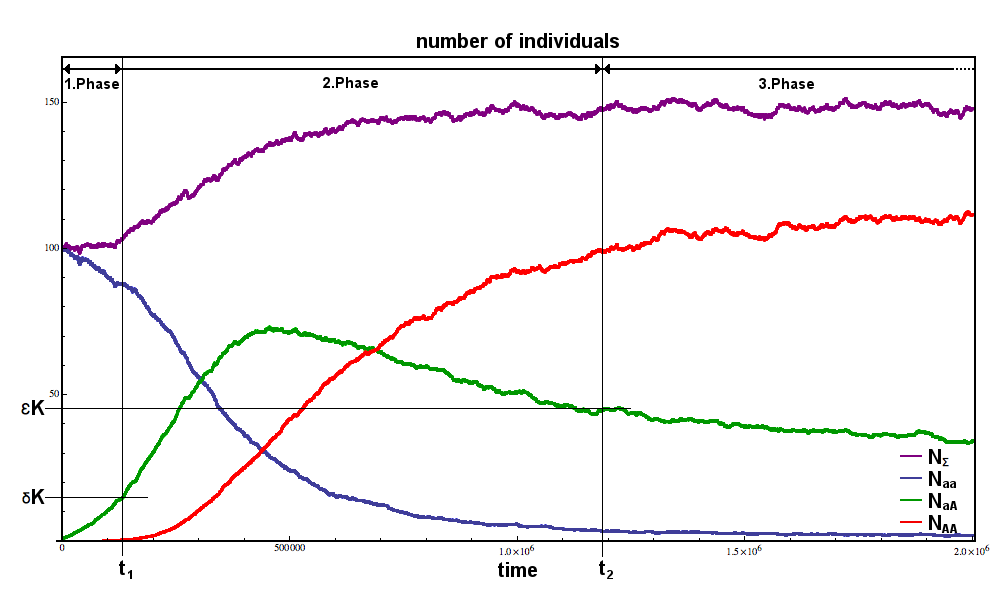}
	\caption{The three phases of the proof \tiny{(Simulation by Loren Coquille)}}
	\label{fig2}
\end{figure}\\
\begin{enumerate}
\item[\textbf{Phase 1:}]Fixation of the mutant population,
\item [\textbf{Phase 2:}]Invasion of the mutant population,
\item[\textbf{Phase 3:}]Survival of the recessive allele.
\end{enumerate}
\vspace{12pt}
The first two phases are similar to the ones in \citen{CMM13}, whereas the last phase will be analysed in ten steps.
Technically, the analysis uses tools developed in \citen{B14} and classical potential methods  (see p.e. \citen{BH15}).\\
\begin{enumerate}
\item[]\hspace{-1.22cm}\textbf{Settings for the steps}
\item[\textbf{Step 1:}]Upper bound on $\S(t)$,
\item[\textbf{Step 2:}]Upper bound on $n_{aa}(t)$,
\item[\textbf{Step 3:}]Lower bound on $\S(t)$,
\item[\textbf{Step 4:}]Upper and lower bound on $n_{AA}(t)$,
\item[\textbf{Step 5:}]Decay of $n_{aA}(t)$,
\item[\textbf{Step 6:}]Decay time of $n_{aA}(t)$,
\item[\textbf{Step 7:}]Decay and decay time of $n_{aa}(t)$,
\item[\textbf{Step 8:}]Growth and growth time of $\S(t)$,
\item[]\hspace{-1.22cm}\textbf{Total decay time of $n_{aA}(t)$}.
\end{enumerate}
\vspace{3pt} 

%---------------------------------------------------------------------------------
\paragraph{Phase 1: Fixation of the mutant population.} 
In the first phase we show that there is a positive probability that the fitter mutant population 
$A(t)\equiv n_{aA}(t)+2 n_{AA}(t)$ fixates in the resident population. More precisely, as long as the mutant population size 
is smaller than a fixed $\d$, the resident $aa$-population stays close to its equilibrium $\bar n_a$ (Proposition \ref{propD}) 
and its dynamics is nearly the same as before since the influence of the mutant population is negligible. We can 
approximate the dynamics of the mutant population $A(t)$ by a birth and death process and can show that the probability 
that this branching process increases to a $\d$-level is close to its survival probability and hence also the probability that 
the mutant population $A(t)\equiv n_{aA}(t)+2 n_{AA}(t)$ grows up to a size $\d$. This is the content of Theorem \ref{fix}.
%---------------------------------------------------------------------------------

\paragraph{Phase 2: Invasion of the mutant population.} 
The fixation (Phase 1) ends  with a macroscopic mutant population of size $\d$. In the second phase the mutant 
population invades the resident population and suppresses it. By the Large Population Approximation 
(Theorem \ref{LPA}, \citen{F_MA}) the behaviour of the process is now close to the solution of the deterministic
 system \eqref{det} with the same initial condition on any finite time interval,  when $K$ tends to infinity. Thus, we get from 
 the analysis of the dynamical system in Section \ref{pardet} that any solution starting in a $\d$-neighbourhood of
  $(\bar n_a,0,0)$ converges to an $\e$-neighbourhood of $(0,0,\bar n_A)$ in finite time ($t_2$ in Figure \ref{fig2}).  
Since we see in the dynamical system that as soon as the $AA$-population is close to its equilibrium, the $aA$-population 
decays like $\frac{1}{t}$, we only proceed until $n_{aA}$ decreases to an $\e$-level to ensure that the duration of this 
phase is still finite. 

In \citen{C06}, it is shown that the duration of the first phase is of order $\OO(\ln K)$ and that the time for the second 
phase is bounded. Thus the time needed by the $aA$-population to reach again the $\e$-level after the fixation is of order 
$\OO(\ln K)$ (cf. Theorem \ref{main1} (i)). From Proposition \ref{propD} we get that the resident $aa$-population stays in a 
$\d$-neighbourhood of its equilibrium $\bar n_a$ an exponentially long time $\exp(VK)$ as long as the mutant population is 
smaller than $\d$. Thus we can approximate the rate of mutation until this exit time by $\mu_KfK\bar n_a$. Hence the 
waiting time for mutation to occur is of order $\frac1{K\mu_K}$.
\citen{C06} proved that there is also no accumulation of mutations in the second phase. More precisely, he shows that, for any initial condition, 
the probability of a mutation on any bounded time interval is very small:
\begin{lemma}[Lemma 2 (a) in \citen{C06}]
\label{lemma2a}
Assume that the initial condition of $\nu_t$ satisfies \\$\sup_K\E(\langle\nu_0,\1\rangle)<\infty$. Then, for any $\eta>0$, 
there exists an $\e>0$ such that, for any $t>0$,
\begin{align}
\limsup_{K\rightarrow\infty}\P_{\nu_0}^K\left(\exists n\geq0:\frac t{K\mu_K}\leq\t_n\leq\frac{t+\e}{K\mu_K}\right)<\eta,
\end{align}
where $\t_n$ are the ordered sequence of times when a mutation occurs, defined in the beginning of Chapter 3.
\end{lemma}
Using Lemma \ref{lemma2a}, we get that, for fixed $\eta>0$, there exists a constant, $\eta>\rho>0$, such that, for sufficiently 
large $K$,
\begin{align}
\label{muttime}
\P_{\frac{z_K}{K}\d_{aa}+\frac{1}{K}\d_{aA}}\left(\t_1<\frac{\rho}{K\mu_K}\right)<\d,
\end{align}
where $\t_1$ is the time of the next mutation. Thus, the next mutation occurs with high probability not before a time 
$\frac{\rho}{K\mu_K}$.
Hence, under the assumption that 
\begin{align}
\ln(K)\ll\frac1{K\mu_K},
\end{align}
(cf. left inequality of \eqref{mu}) there appears no mutation before the first and second phase are completed.
%---------------------------------------------------------------------------------

\paragraph{Phase 3: Survival of the recessive allele.}
The last phase starts as soon as $n_{aA}(t)$ has decreased to an $\e$-level. This phase is different from the one in 
\citen{C06} and \citen{CMM13}, since the analysis of the deterministic system in Section \ref{pardet} reveals  that 
$n_{aA}(t)$ decreases only 
 like a function $f(t)=\frac{1}{t}$, in contrast to the exponential decay in \citen{CMM13}. Thus, we may expect that 
 the time to extinction will not be $\OO(\ln K)$ anymore, and the recessive allele $a$ will survive in the population
  for a much longer time. 
This is a situation similar to the one encountered in \citen{B14}, where it was necessary to show that the stochastic 
system remains close to a deterministic one over times of order $K^\a$ due to the fact that the evolutionary advantage 
of the mutant population vanishes like a negative power of $K$. 
Just as in that case, we  cannot use the Law of Large Numbers, but we adopt  the stochastic Euler scheme from \citen{B14}
 to show that the behaviour of the deterministic and the stochastic system remain close for a time of order $K^{1/4-\a}$.\\

Let us put this scheme on a mathematically footing:\\
\begin{figure}[b]
	\centering
  \includegraphics[width=1.0\textwidth]{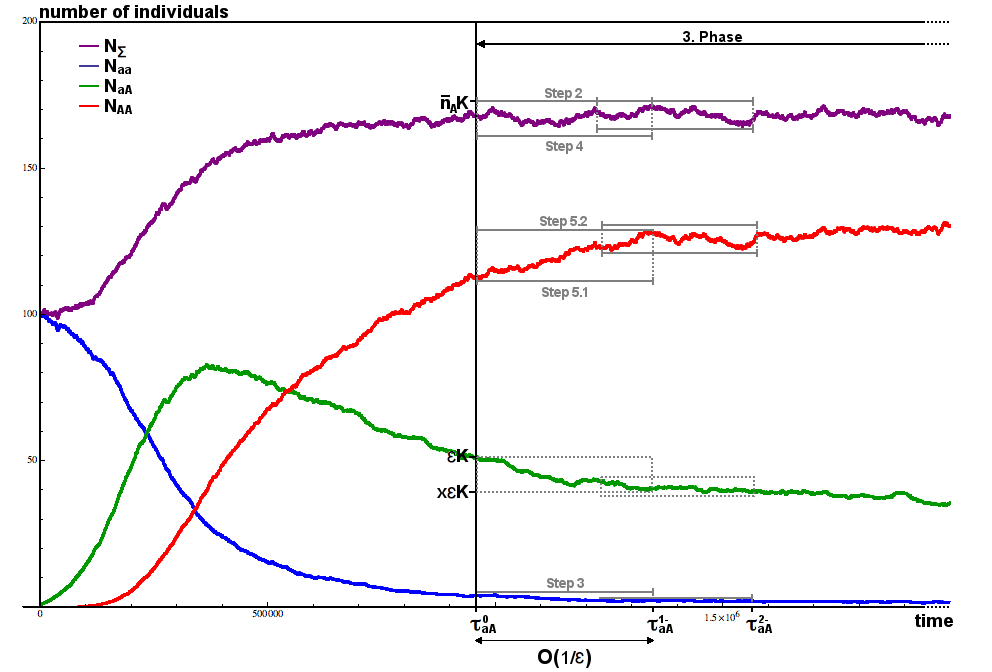}
	\caption{Steps of the proof \tiny{(Simulation by Loren Coquille)}}
	\label{fig4}
\end{figure}
%---------------------------------------------------------------------------------

\paragraph{Settings for the steps.} We define  stopping times depending on 
 $n_{aA}(t)$ in
 such a way that we can control the other processes $n_{aa}(t), n_{AA}(t)$, and $\S(t)$ on the resulting time intervals.
Fix $\e>0$ and $\vartheta>0$ such that $\e<\frac{\D}{2}<\vartheta<\D$. We set
\begin{equation}
\label{defx}
	x=\left(\frac{f+\vartheta}{f+\D}\right)^{\frac{1}{2}},
\end{equation}
and, for $0\leq i\leq\left\lfloor\frac{-\ln(\e K^{1/4-\a})}{\ln(x)}\right\rfloor$, with $\a>0$, we define the stopping times on $n_{aA}(t)$ by
\begin{align}
	&\tau_{aA}^{i+}\equiv\inf\left\{t\geq \tau_{aA}^{i-}: n_{aA}(t)\geq x^i\epsilon+x^{2i}\epsilon^2\right\},\label{stoppaA+}\\
	&\tau_{aA}^{i-}\equiv\inf\left\{t\geq \tau_{aA}^{(i-1)-}: n_{aA}(t)\leq x^i\epsilon\right\}.\label{stoppaA-}
\end{align}
During the time intervals $t\in \left[\t_{aA}^{i-}, \t_{aA}^{i+}\land\t_{aA}^{(i+1)-}\right]$, 
 $n_{aA}(t)\in\left[x^{i+1}\e,x^i\e+x^{2i}\e^2\right]$. 
The upper bound on $i$ is chosen in such a way that 
\begin{align}
\label{ubi}
x^i\e\geq K^{-1/4+\a}.
\end{align}
The following eight steps will be iterated from $i=0$ to $i=\left\lfloor\frac{-\ln(\e K^{1/4-\a})}{\ln(x)}\right\rfloor$.
\begin{remark}
Since in Phase 3 the biggest contribution to the birth rate of $n_{aa}(t)$ is given by the combination of two $aA$-individuals, 
$n_{aa}$ behaves like $n_{aA}^2$. We let $n_{aA}$ decrease only until $K^{-1/4+\a}$. Afterwards $n_{aa}$ would be of smaller order than $K^{-1/2}$ and the natural fluctuations of the big $AA$-population, of order $K^{-1/2}$, would induce the death of the $aa$-population due to competition. Since $n_{aa}$ contributes to the birth of the $aA$-population we also loose the control over this.
\end{remark}
%---------------------------------------------------------------------------------
\paragraph{Step 1: Upper bound on $\S(t)$.} We show that, on the time interval $t\in \left[\t_{aA}^{i-}, \t_{aA}^{i+}\land\t_{aA}^{(i+1)-}\land 
e^{VK^{\a}}\right]$, there exists a constant, $M_\S>0$, such that the sum process $\S(t)$
does not exceed the level $\bar n_A+3M_\S (x^{2i}\e^2)^{1+\a}$, with high probability: 
\begin{proposition}
\label{upSum}
For all $M>0$ and $0\leq i\leq\left\lfloor\frac{-\ln(\e K^{1/4-\a})}{\ln(x)}\right\rfloor$, let
	\begin{equation}
		\t_{\S,M}^{\a}\equiv\inf\left\{t>\t_{aA}^{i-}: \S(t)-\bar n_A\geq 3M(x^{2i}\e^2)^{1+\a}\right\}.
	\end{equation}
Then there exists a constant $M_{\S}>0$ such that
	\begin{equation}
		\P\left[\t_{\S,M_\S}^{\a}<\t_{aA}^{i+}\land\t_{aA}^{(i+1)-}\land e^{VK^{\a}}\right]=o(K^{-1}).
	\end{equation}
\end{proposition}
To prove Proposition \ref{upSum}, we define the difference process between $\S(t)K$ and $\bar n_AK$  and couple it to
 a birth-death-immigration process. 
We show that this process jumps up with probability less than $\frac12$  and show that the probability that the process 
reaches the level $3M_\S (x^{2i}\e^2)^{1+\a}K$ before going to zero, is very 
small. Then we show that the process returns many times to zero until it reaches the level $3M_\S  (x^{2i}\e^2)^{1+\a}K$ and 
calculate the time for one such return.
\begin{remark}
This is only a coarse bound on the sum process $\S(t)$ but with our initial conditions we are not able to get a finer one. After Step 7 we have enough information to refine it but for the iteration this upper bound suffices.
\end{remark}
%---------------------------------------------------------------------------------

\paragraph{Step 2: Upper bound on $n_{aa}(t)$.} An upper bound on $n_{aa}(t)$ is obtained  similarly 
as in Step 1.
Let
\begin{align}
\label{gammadelta}
\g_\D\equiv\frac{f+\frac\D2}{4\bar n_A(f+\D)}.
\end{align}
We show that, on the time interval $t\in \left[\t_{aA}^{i-}, \t_{aA}^{i+}\land\t_{aA}^{(i+1)-}\land e^{VK^{\a}}\right]$, 
there exists a constant, $M_{aa}>0$, such that the $aa$-population does 
not exceed the level $\g_\D x^{2i}\e^2+3M_{aa} (x^{2i}\e^2)^{1+\a}$, with high probability:
\begin{proposition}
\label{ubaa}
For all $M>0$ and $0\leq i\leq\left\lfloor\frac{-\ln(\e K^{1/4-\a})}{\ln(x)}\right\rfloor$, let
	\begin{equation}
		\t_{aa,M}^{\a}\equiv\inf\left\{t>\t_{aA}^{i-}: n_{aa}(t)-\g_\D x^{2i}\e^2\geq 3M(x^{2i}\e^2)^{1+\a}\right\}.
	\end{equation}
Then there exists a constant, $M_{aa}>0$, such that
	\begin{equation}
		\P\left[\t_{aa,M_{aa}}^{\a}<\t_{aA}^{i+}\land\t_{aA}^{(i+1)-}\land e^{VK^{\a}}\right]=o(K^{-1}).
	\end{equation}
\end{proposition}
The proof is similar to the one in Step 1. We define the difference process between $n_{aa}(t)K$ and 
$\g_\D x^{2i}\e^2K$ and couple it to a birth-death-immigration process. We show that this branching process jumps up with probability less than $\frac12$. 
Again,  we show that the probability that the process reaches the level $3M_{aa}(x^{2i}\e^2)^{1+\a}K$ before going to zero, is very small. Then we show that the process returns many times to zero until it reaches the level $3M_{aa} (x^{2i}\e^2)^{1+\a}K$ and calculate the time for one such return.
%---------------------------------------------------------------------------------

\paragraph{Step 3: Lower bound on $\S(t)$.} With the results  from the previous steps, we can bound $\S(t)$ from below. 
We show that on the time interval $t\in \left[\t_{aA}^{i-}, \t_{aA}^{i+}\land\t_{aA}^{(i+1)-}\land e^{VK^{\a}}\right]$ the sum 
process does not drop below $\bar n_A-\frac{\D+\vartheta}{c\bar n_A}\g_\D x^{2i}\e^2-3M_\S(x^{2i}\e^2)^{1+\a}$, with high 
probability:
\begin{proposition}
\label{lbSum}
For all $M>0$ and $0\leq i\leq\left\lfloor\frac{-\ln(\e K^{1/4-\a})}{\ln(x)}\right\rfloor$, let
	\begin{equation}
		\t_{\S,M}^{\a}\equiv\inf\left\{t>\t_{aA}^{i-}: \S(t)-\left(\bar n_A-\frac{\D+\vartheta}{c\bar n_A}\g_\D x^{2i}\e^2 \right)\leq -3M(x^{2i}\e^2)^{1+\a}\right\}.
	\end{equation}
Then there exists a constant, $M_{\S}>0$, such that
	\begin{equation}
		\P\left[\t_{\S,M_{\S}}^{\a}<\t_{aA}^{i+}\land\t_{aA}^{(i+1)-}\land e^{VK^{\a}}\right]=o(K^{-1}).
	\end{equation}
\end{proposition}
The proof is similar to those in Step 1 and 2.
%---------------------------------------------------------------------------------

\paragraph{Step 4: Lower and upper bound on $n_{AA}(t)$.} Since we now have bounded the processes 
$n_{aa}(t), n_{aA}(t)$ and $\S(t)$ from above and below (for $n_{aa}(t)$ it suffices to set the lower bound to zero), it is easy to get a lower and an upper bound on 
$n_{AA}(t)$ on the time interval $t\in \left[\t_{aA}^{i-}, \t_{aA}^{i+}\land\t_{aA}^{(i+1)-}\land e^{VK^{\a}}\right]$. 
There exists a constant, $M_{AA}>0$, such that with high probability $n_{AA}(t)$ does not drop below $\bar n_{A}-x^i\e-M_{AA}x^{2i}\e^2$
 (Proposition \ref{lbAA}), and does not exceed the level $\bar n_{A}-x^{(i+1)}\e-M_{AA}(x^{2i}\e^2)^{1+\a}$
  (Proposition \ref{ubAA}):
\begin{proposition}
\label{lbAA}
For all $M>0$ and $0\leq i\leq\left\lfloor\frac{-\ln(\e K^{1/4-\a})}{\ln(x)}\right\rfloor$, let
	\begin{equation}
		\t_{AA,M}^{2i}\equiv\inf\left\{t>\t_{aA}^{i-}: n_{AA}(t)-\left(\bar n_A-x^i\e\right)\leq -Mx^{2i}\e^2\right\}.
	\end{equation}
Then there exists a constant $M_{AA}>0$ such that
	\begin{equation}
		\P\left[\t_{AA,M_{AA}}^{2i}<\t_{aA}^{i+}\land\t_{aA}^{(i+1)-}\land e^{VK^{\a}}\right]=o(K^{-1}).
	\end{equation}
\end{proposition}
\begin{proposition}
\label{ubAA}
For all $M>0$ and $0\leq i\leq\left\lfloor\frac{-\ln(\e K^{1/4-\a})}{\ln(x)}\right\rfloor$, let
	\begin{equation}
		\t_{AA,M}^{\a}\equiv\inf\left\{t>\t_{aA}^{i-}: n_{AA}(t)-\left(\bar n_A-x^{i+1}\e\right)\geq M(x^{2i}\e^2)^{1+\a}\right\}.
	\end{equation}
Then there exists a constant, $M_{AA}>0$, such that
	\begin{equation}
		\P\left[\t_{AA,M_{AA}}^{\a}<\t_{aA}^{i+}\land\t_{aA}^{(i+1)-}\land e^{VK^{\a}}\right]=o(K^{-1}).
	\end{equation}
\end{proposition}
%---------------------------------------------------------------------------------
\paragraph{Step 5: Decay of $n_{aA}(t)$.} We now have upper and lower bounds for all the single processes, 
for $t\in \left[\t_{aA}^{i-}, \t_{aA}^{i+}\land\t_{aA}^{(i+1)-}\land e^{VK^{\a}}\right]$. Using these bounds, we 
 prove that $n_{aA}(t)$ has the tendency to decrease on a given time interval. We show that, with high probability, 
 $n_{aA}(t)$, restarted at $x^i\e$ (i.e. we set $\tau_{aA}^{i-}=0$), hits the level $x^{i+1}\e$ before it reaches the level $x^i\e+x^{2i}\e^2$.
\begin{proposition}
\label{decayaA}
	There exists a constant $C>0$ such that, forall $0\leq i\leq\left\lfloor\frac{-\ln(\e K^{1/4-\a})}{\ln(x)}\right\rfloor$
		\begin{align}
			\P\left[\t_{aA}^{i+}<\t_{aA}^{(i+1)-}\big|n_{aA}(0)=x^i\e\right]\leq K^{1/4-\a}\exp\left(-CK^{1/4+3\a}\right).
		\end{align}
\end{proposition}
For the proof we couple $n_{aA}(t)$ to majorising and minorising birth-death-immigration processes  and show that these processes jump up with
 probability less than $\frac12$. This way we prove that with high probability $n_{aA}(t)$ reaches $x^{i+1}\e$ before going back to $x^{i}\e+x^{2i}\e^2$.
 
%---------------------------------------------------------------------------------
\paragraph{Step 6: Decay time of $n_{aA}(t)$.} This is the step where we see that $n_{aA}(t)$ decays like a function $f(t)=\frac{1}{t}$. Precisely, it is shown that the time which the $aA$-population needs to decrease from $x^{i}\e$ to $x^{i+1}\e$ is of order $\frac{1}{x^i\e}$:
\begin{proposition}
\label{timeaA}
	Let 
		\begin{align}
			\theta_i(aA)\equiv\inf\left\{t\geq0: n_{aA}(t)\leq x^{i+1}\e\big|n_{aA}(0)=x^{i}\e\right\},
		\end{align}
	the decay time of $n_{aA}(t)$ on the time interval $t\in \left[\t_{aA}^{i-}, \t_{aA}^{i+}\land\t_{aA}^{(i+1)-}\land e^{VK^\a}\right]$. Then 
	for all $0\leq i\leq\left\lfloor\frac{-\ln(\e K^{1/4-\a})}{\ln(x)}\right\rfloor$   there exist finite, positive constants, $C_l,C_u$, and a constant $M>0$, such that 
	\begin{align}
		\P\left[ \frac {C_u}{x^i\e}\geq  \theta_i(aA)\geq \frac{C_l}{x^i\e}\right]\geq1-\exp\left(-MK^{1/2+2\a}\right).
	\end{align}
\end{proposition}
To prove this proposition we calculate an upper bound on the decay time of the majorising process obtained in Step 5 and a lower bound on the decay time of the minorising process of the same order. 
Precisely, we estimate the number of jumps  the processes make until they reach $x^{i+1}\e$, and the time of one jump.
%---------------------------------------------------------------------------------

\paragraph{Step 7: Decay and decay time of $n_{aa}(t)$.} To carry out the iteration,  we have to ensure that, 
on a given time interval, the $aa$-population decreases below the upper bound needed for the next iteration step. 
We show that $n_{aa}(t)$ decreases from $\g_\D x^{2i}\e^2+M_{aa}(x^{2i}\e^2)^{1+\a}$ to $\g_\D x^{2i+2}\e^2$, and stays smaller than 
$\g_\D x^{2i+2}\e^2+M_{aa}(x^{2i+2}\e^2)^{1+\a}$ when $n_{aA}(t)$ reaches $x^{i+1}\e$:
\begin{proposition}
\label{timeaa}
	For $t\in \left[\t_{aA}^{i-}, \t_{aA}^{i+}\land\t_{aA}^{(i+1)-}\land e^{VK^\a}\right]$, $n_{aa}(t)$ moves from $\g_\D x^{2i}\e^2+M_{aa}(x^{2i}\e^2)^{1+\a}$ to $\g_\D x^{2i+2}\e^2$ 
	 the process $n_{aA}$ decreases  from $x^{i}\e$ to $x^{i+1}\e$, and stays below $\g_\D x^{2i+2}\e^2+M_{aa}(x^{2i+2}\e^2)^{1+\a}$ until $n_{aA}(t)$ hits the $x^{i+1}\e$-level. 
\end{proposition}
The proof of this proposition has three parts:
First, as in Step 5, we show that $n_{aa}(t)$ has the tendency to decrease and that it reaches $\g_\D x^{2i+2}\e^2$ before 
going back to $\g_\D x^{2i}\e^2+M_{aa}(x^{2i}\e^2)^{1+\a}$.
The second part is similar to Step 6, where we estimate the number of jumps and the duration of  one jump of the process. 
In the last part we show, as in Step 2, that the process stays below $\g_\D x^{2i+2}\e^2+M_{aa}(x^{2i+2}\e^2)^{1+\a}$ until $n_{aA}(t)$ hits the level $x^{i+1}\e$ and the next iteration step starts.
%---------------------------------------------------------------------------------

\paragraph{Step 8: Growth and growth time of $\S(t)$.} Similarly to Step 7, we also have to ensure that the sum process 
increases from the level $\bar n_A-\frac{\D+\vartheta}{c\bar n_A}\g_\D x^{2i}\e^2-M_{\S}(x^{2i}\e^2)^{1+\a}$ to 
$\bar n_A-\frac{\D+\vartheta}{c\bar n_A}\g_\D x^{2i+2}\e^2$ on a given time interval and is greater than 
$\bar n_A-\frac{\D+\vartheta}{c\bar n_A}\g_\D x^{2i+2}\e^2-M_{\S}(x^{2i+2}\e^2)^{1+\a}$ when the $aA$-population reaches the 
level $x^{i+1}\e$:
\begin{proposition}
\label{timeSum}
	While  $n_{aA}$ decreases
	 from $x^{i}\e$ to $x^{i+1}\e$, the sum process $\S(t)$ increases from \\
	 $\bar n_A-\frac{\D+\vartheta}{c\bar n_A}\g_\D x^{2i}\e^2-M_{\S}(x^{2i}\e^2)^{1+\a}$ to $\bar n_A-\frac{\D+\vartheta}{c\bar n_A}\g_\D x^{2i+2}\e^2$ and stays above 
	 $\bar n_A-\frac{\D+\vartheta}{c\bar n_A}\g_\D x^{2i+2}\e^2-
	 M_{\S}(x^{2i+2}\e^2)^{1+\a}$ until the $aA$-population hits the
	  $x^{i+1}\e$-level.
\end{proposition}
The proof uses the same three parts as described in the proof of Proposition \ref{timeaa}.
%---------------------------------------------------------------------------------

\paragraph{\textbf{Total decay time of $n_{aA}(t)$}.} We iterate Step 1 to 8 until $i=\left\lfloor\frac{-\ln(\e K^{1/4-\a})}{\ln(x)}\right\rfloor$, 
the value for which $n_{aA}(t)$ is of order $K^{-1/4+\a}$. Finally, we sum up the decay time of the $aA$-population in each 
iteration step and get the desired result (Theorem \ref{main1} (ii)). \\
Moreover, we ensure the upper bound on the mutation probability $\mu_K$ in Theorem \ref{main2}.
%-----------------------------------
\section{The deterministic system}
\label{pardet}

\subsection{The large population approximation}
The main ingredient for the second phase is the deterministic system, since we know from \citen{F_MA} or \citen{CMM13} that, for large populations, the behaviour of the stochastic
 process is close to the solution of a deterministic equation. Thus we analyse it here.
\begin{proposition}[Proposition 3.2 in \citen{CMM13}]
\label{LPA}
Let $T>0$ and $C\subset\R_+^3$ compact.
Assume that the initial condition $\frac{1}{K}(N_{aa}^0,N_{aA}^0,N_{AA}^0)$ converges almost surely to a deterministic vector $(x_0,y_0,z_0)\in C$ when $K$ goes to infinity. 
Let $(x(t),y(t),z(t))=\phi(t;(x_0,y_0,z_0))$ denote the solution to
\begin{align}
\label{det}
\dot\phi(t;(x_0,y_0,z_0))=\left(
\begin{array}{c}
\tilde b_{aa}(x(t),y(t),z(t))-\tilde d_{aa}(x(t),y(t),z(t))\\
\tilde b_{aA}(x(t),y(t),z(t))-\tilde d_{aA}(x(t),y(t),z(t))\\
\tilde b_{AA}(x(t),y(t),z(t))-\tilde d_{AA}(x(t),y(t),z(t))
\end{array}
\right)=:X(x(t),y(t),z(t)),
\end{align}
where 
\begin{align*}
\tilde b_{aa}(x(t),y(t),z(t))&=\frac{(f_{aa}x(t)+\sfrac{1}{2}f_{aA}y(t))^2}{(f_{aa}x(t)+f_{aA}y(t)+f_{AA}z(t))},\\
\tilde d_{aa}(x(t),y(t),z(t))&=x(t)(D_{aa}+c_{aa,aa}x(t)+c_{aa,aA}y(t)+c_{aa,AA}z(t)),
\end{align*}
and similar expression for the $aA$- and $AA$-type. Then, forall $T>0$,
\begin{align}
\label{inva}
\lim_{K\rightarrow\infty}\sup_{t\in[0,T]}|n_{uv}(t)-\phi_{uv}(t;(n_{aa}^0,n_{aA}^0,n_{AA}^0))|=0\quad, a.s.,
\end{align}
for all $uv\in\{aa, aA, AA\}$.
\end{proposition}
Thus, to understand the behaviour of the process we have to analyse the deterministic system \eqref{det} above. The vector field \eqref{det} of the model we consider is given by
\begin{align}
\label{XD}
X(x,y,z)=X_\D(x,y,z)=
\begin{pmatrix}
f\frac{(x+\frac{1}{2}y)^2}{x+y+z}-(D+\D+c(x+y+z))x\\
2f\frac{(x+\frac{1}{2}y)(z+\frac{1}{2}y)}{x+y+z}-(D+c(x+y+z))y\\
f\frac{(z+\frac{1}{2}y)^2}{x+y+z}-(D+c(x+y+z))z
\end{pmatrix},
\end{align}
which has some particular properties:
\begin{theorem}
\label{dettheo}
Assume \textbf{(A)}+\textbf{(B)} and let $\e>0$, then
\begin{itemize}
\item[(i)] the vector field \eqref{XD} has the unstable fixed point $\mathfrak n_{aa}\equiv(\bar n_a,0,0)$ and the stable fixed point $\mathfrak n_{AA}\equiv(0,0,\bar n_A)$,
\item[(ii)] the Jacobian matrix at the unstable fixed point $\mathfrak n_{aa}$, $DX_\D(\mathfrak n_{aa})$, has two negative and one positive eigenvalues,
\item[(iii)] the Jacobian matrix at the stable fixed point $\mathfrak n_{AA}$, $DX_\D(\mathfrak n_{AA})$, has two negative and one zero eigenvalues,
\item[(iv)] for $\varrho<f\D$, and as soon as the $aA$-population decreased to an $\e$-level, then
\begin{align}
\frac{2\bar n_A(f+\D)}{(f\D+\varrho) t+\sfrac{2\bar n_A(f+\D)}{\e}}\leq n_{aA}(t)\leq\frac{2\bar n_A(f+\D)}{(f\D-\varrho) t+\sfrac{2\bar n_A(f+\D)}{\e}}.
\end{align}
\end{itemize}
\end{theorem}
There is also a biological explanation for the behaviour of $n_{aA}(t)$ described in Theorem \ref{dettheo} (iv). 
Since the $A$-allele is the fittest and dominant one and because of the phenotypic viewpoint the $aA$-population is as fit 
as the $AA$-population and both die with the same rate. The $aA$-population only decreases because of the 
disadvantage in reproduction due to the less fit, decreasing $aa$-population.
Observe that Theorem \ref{dettheo} (i)+(ii) also holds in the model of \citen{CMM13} (cf. Proposition 3.3 therein) but the 
Jacobian matrix of their fixed point $\mathfrak n_{AA}$ has three negative eigenvalues and thus they get the exponential
decay of $n_{aA}(t)$.\\
The behaviour of solutions of the deterministic system can be analysed  using the following result of
\citen{CMM13}:
\begin{theorem}[Theorem C.2 in \citen{CMM13}]
\label{C2}
Let $\z=u_A-u_a$ be the variation of the allelic trait. Suppose it is non zero and of small enough modulus. 
If $\z\frac{dS_{aA,aa}}{d\z}(0)>0$ then the fixed point $\mathfrak n_{aa}$ is unstable and we have fixation for the 
macroscopic dynamics.\\
More precisely, there exists an invariant stable curve $\G_{\z}$ which joins $\mathfrak n_{aa}$ to $\mathfrak n_{AA}$. 
Moreover there exists an invariant tubular neighbourhood $\VV$ of $\G_\z$ such that the orbit of any initial condition in 
$\VV$ converges to $\mathfrak n_{AA}$.\\
If $\z\frac{dS_{aA,aa}}{d\z}(0)<0$ the fixed point $\mathfrak n_{aa}$ is stable and the mutant disappears in the 
macroscopic dynamics.
\end{theorem}
Their proof works as follows. First they consider the unperturbed version $X_0$ of the vector field \eqref{det} in the case of 
neutrality between the alleles $A$ and $a$. That is $f_{u_1u_2}=f$, $D_{u_1u_2}=D$ and $C_{u_1u_2,v_1v_2}=c$, for 
$u_1u_2,v_1v_2\in\{aa,aA,AA\}$. They get that this system has a line of fixed points $\G_0$ which is transversally 
hyperbolic. Afterwards they consider the system $X_\z$ with small perturbations $\z$. 
From Theorem 4.1 in \citen{H77} (p. 39) they deduce that there exists an attractive and invariant curve $\G_\z$, converging to 
$\G_0$, as $\z\rightarrow0$.
Hence, there is a small enough tubular neighbourhood $\VV$ of $\G_0$ such that $\G_\z$ is contained in $\VV$ and 
attracts all orbits with initial conditions in $\VV$ (cf. Figure \ref{fig2}).
To show that the orbit of any initial condition on $\G_\z$ converges to one of the two fixed points $\mathfrak n_{aa}$ or 
$\mathfrak n_{AA}$, one have to ensure that the vector field does not vanish on $\G_\z$, except for these two fixed 
points. Since the curve $\G_\z$ is attractive, it is equivalent to look for the fixed points in the tubular $\VV$. 
For finding the zero points in $\VV$, \citen{CMM13} use linear combinations of the left eigenvectors of $DX_0(\G_0(v))^t$, 
with the perturbed vector field $X_\z$. First they quote to zero the two linear combinations with the eigenvectors which 
span the stable affine subspace. By the implicit function Theorem, they get a curve which contains all possible zeros in 
$\VV$ (cf. Proposition B.2 in \citen{CMM13}). Then they consider the last linear combination and look for the points on the 
received curve where it vanish. Under the conditions that the derivative of the third linear combination at the point 
$\bar n_A$ is non zero and does not vanishes between the fixed points $\pm\bar n_A$ they get that $X_\z$ has only 
two zeros in a tabular neighbourhood of $\G_0$ (cf. Theorem B.4 in \citen{CMM13}).\\
We have to do some extra work to get the same result for the model with dominant $A$-allele since the derivative of the 
third linear combination in our model, described above, is zero at the point $\bar n_A$. But by an easy calculation (see \eqref{easycal1} and \eqref{easycal2}) we can 
indeed prove that this point is an isolated zero and we can deduce from Theorem \ref{C2} the following corollary, which is the main result we need about the dynamical system.
\begin{corollary}
\label{fp}
Let $\D\neq0$ small enough. 
\begin{itemize}
\item[(i)] The attracting  and invariant curve $\Gamma_\D$ of the perturbed vector field $X_\D$ \eqref{det} contained in the 
positive quadrant, is the piece of unstable manifold between the equilibrium points $\mathfrak n_{aa}$ and 
$\mathfrak n_{AA}$.
\item[(ii)] There exists an invariant tubular neighbourhood $\VV$ of $\Gamma_\D$ such that the orbit of any initial condition 
in $\VV$ converges to the equilibrium point $\mathfrak n_{AA}$.
\end{itemize}
\end{corollary} 
Hence, if we start the process in a neighbourhood of the unstable fixed point $\mathfrak n_{aa}$, it will leave this 
neighbourhood in finite time and converge to a neighbourhood of the stable fixed point $\mathfrak n_{AA}$.\\

\begin{figure}[t]
	\centering
\includegraphics[width=0.5\textwidth]{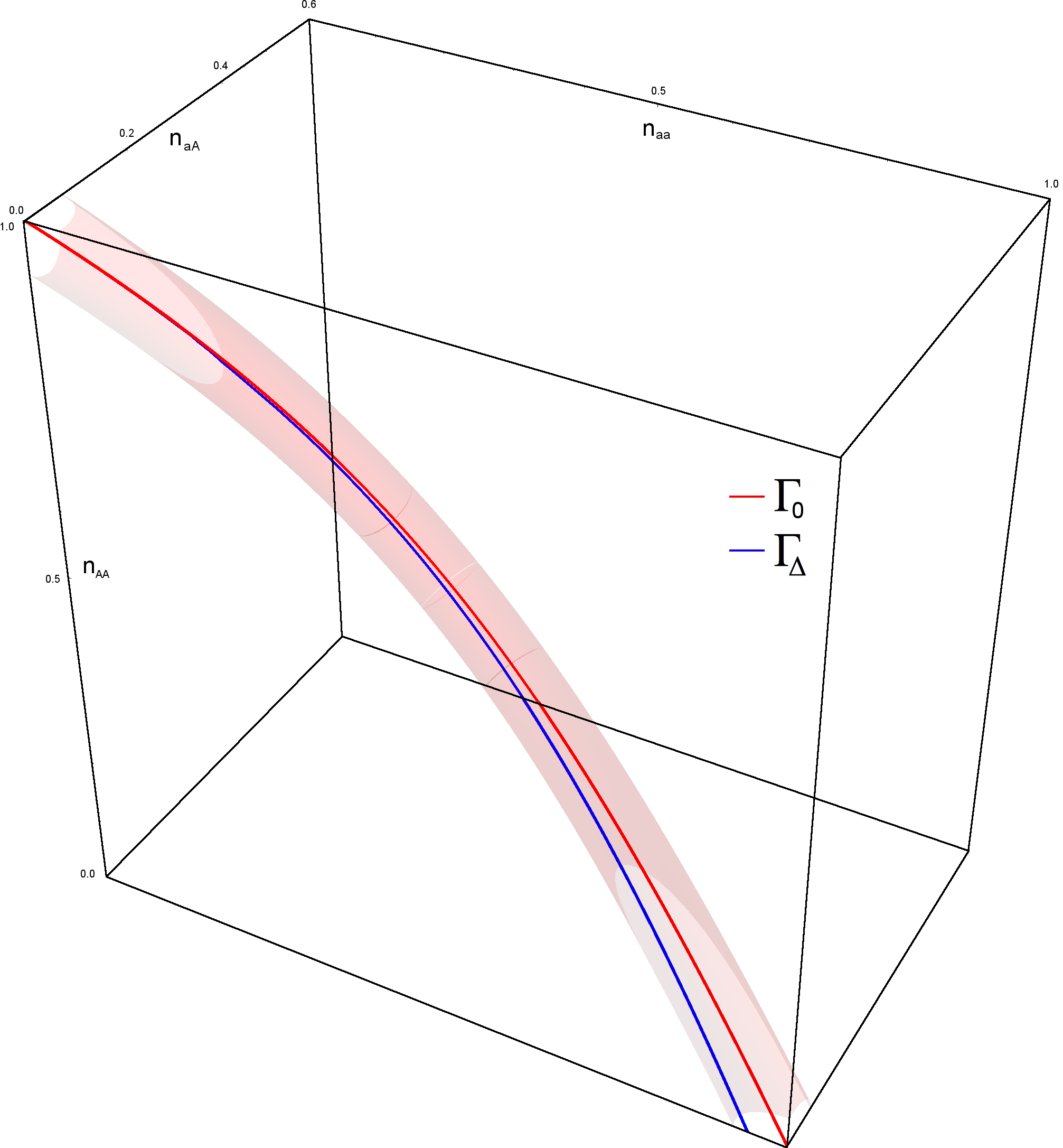}
	\caption{The curves $\G_0$, $\G_\D$ and the tube $\VV$ in the perturbed vector field $X_\D$ \tiny{(Simulation by Loren Coquille)}}
	\label{fig3}
\end{figure}
%--------------------------------------------------------------------------------
\section{Proofs of Theorem \ref{dettheo} and the main theorems}
%---------------------------------------------------------------------------------
\subsection{Analysis of the deterministic system}
\label{detproof}
Because of Proposition \ref{LPA} we have to analyse the deterministic system \eqref{det} 
 (a simulation is shown in Figure \ref{fig1}).
 %------------------------
\begin{proof}[Proof of Theorem \ref{dettheo}]
In the following we consider the differential equations of $n_{aa}(t), n_{aA}(t)$ and $n_{AA}(t)$, given by \eqref{XD}:
\begin{align}
\label{diffeq1}
\dot n_{aa}(t)&=f\frac{\left(n_{aa}(t)+\frac12n_{aA}(t)\right)^2}{\S(t)}-n_{aa}(t)(D+\D+c\S(t)),\\
\label{diffeq2}
\dot n_{aA}(t)&=2f\frac{\left(n_{aa}(t)+\frac12n_{aA}(t)\right)\left(n_{AA}(t)+\frac12n_{aA}(t)\right)}{\S(t)}-n_{aA}(t)(D+c\S(t)),\\
\label{diffeq3}
\dot n_{AA}(t)&=f\frac{\left(n_{AA}(t)+\frac12n_{aA}(t)\right)^2}{\S(t)}-n_{AA}(t)(D+c\S(t)).
\end{align}
\paragraph{The Fixed Points: }
By summing \eqref{diffeq1} to \eqref{diffeq3} first, and two times \eqref{diffeq3} and \eqref{diffeq2}, we see that the vector field \eqref{XD} vanishes for the points $\mathfrak n_{aa}$ and $\mathfrak n_{AA}$. 
The Jacobian matrix at the fixed point $\mathfrak n_{aa}$ is given by
\begin{align}
\label{DXaa}
DX_\D((\bar n_a,0,0))=
\begin{pmatrix}
-f+D+\D & -f+D+\D & -2f+D+\D\\
0 & \D & 2f\\
0 &  0 & -f+\D
\end{pmatrix}.
\end{align}
The matrix has the three eigenvalues $\lambda_1=-(f-D-\D)$, $\lambda_2=\D$ and $\lambda_3=-(f-\D)$. For $\D$ small enough and from Assumption \textbf{(A2)} we know that $\lambda_1,\lambda_3<0$, whereas $\lambda_2>0$. Thus the fixed point $\mathfrak n_{aa}$ is unstable.\\
The Jacobian matrix at the fixed point $\mathfrak n_{AA}$ is given by
\begin{align}
DX_\D((0,0,\bar n_A))=
\begin{pmatrix}
 -f-\D &  0 &  0\\
2f & 0 & 0\\
-2f+D & -f+D & -f+D
\end{pmatrix},
\end{align}
it has the three eigenvalues $\lambda_1=-f-\D<0$, $\lambda_2=0$ and $\lambda_3=-(f-D)<0$, from Assumption \textbf{(A2)}. 
The fact that one of the eigenvalues is zero is the main 
novel feature of this system compared to that count in \citen{CMM13}. 
Because of the zero eigenvalue, $\mathfrak n_{AA}$ is a non-hyperbolic equilibrium point of the system and linearization fails to determine its stability properties. Instead, we use the result of the center manifold theory (\citen{H77, P01}) that asserts that the qualitative behaviour of the dynamical system in a neighbourhood of the non-hyperbolic critical point $\mathfrak n_{AA}$ is determined by its behaviour on the center manifold near $\mathfrak n_{AA}$.

\begin{theorem}[The Local Center Manifold Theorem 2.12.1 in \citen{P01}] Let $f\in C^r(E)$, where $E$ is an open subset of $\R^n$ containing the origin and $r\geq1$. Suppose that $f(0)=0$ and $Df(0)$ has $c$ eigenvalues with zero real parts and $s$ eigenvalues with negative real parts, where $c+s=n$. Then the system $\dot z=f(z)$ can be written in diagonal form
\begin{align}
\dot x&=Cx+F(x,y)\nonumber\\
\dot y&=Py+G(x,y),
\end{align}
where $z=(x,y)\in\R^c\times\R^s$, $C$ is a $c\times c$-matrix with $c$ eigenvalues having zero real parts, $P$ is a $s\times s$-matrix with $s$ eigenvalues with negative real parts, and $F(0)=G(0)=0, DF(0)=DG(0)=0.$ Furthermore, there exists $\delta>0$ and a function, $h\in C^r(N_\d(0))$, where $N_\d(0)$ is the $\d$-neighbourhood of $0$, that defines the local center manifold and satisfies:
\begin{align}
\label{cme}
Dh(x)[Cx+F(x,h(x))]-Ph(x)-G(x,h(x))=0,
\end{align}
for $|x|<\d$. The flow on the center manifold $W^c(0)$ is defined by the system of differential equations
\begin{align}
\dot x=Cx+F(x,h(x)),
\end{align}
for all $x\in\R^c$ with $|x|<\d$.
\end{theorem}

The fact that the center manifold of our system near $\mathfrak n_{AA}$ has dimension one, simplifies the problem of determining the stability and the qualitative behaviour of the flow on it near the non-hyperbolic critical point. The Local Center Manifold Theorem shows that the non-hyperbolic critical point $\mathfrak n_{AA}$ is indeed a stable fixed point and that the flow on the center manifold near the critical point behaves like a function $\frac1t$.
This can be seen as follows: Assume that $n_{aA}(t)$ has 
decreased to a level  $\e$. Let $n_{AA}(t)=z(t), n_{aA}(t)=y(t)$ and $n_{aa}(t)=x(t)$.
By the affine transformation $n_{AA}\mapsto n_{AA}-\bar n_A$ we get a translated system
\begin{align}
\label{YD}
Y(z,y,x)=
\begin{pmatrix}
f\frac{(z+\bar n_A+\frac{1}{2}y)^2}{z+y+x+\bar n_A}-(D+c(z+y+x+\bar n_A))(z+\bar n_A)
\\
2f\frac{(z+\bar n_A+\frac{1}{2}y)(x+\frac{1}{2}y)}{z+y+x+\bar n_A}-(D+c(z+y+x+\bar n_A))y\\
f\frac{(x+\frac{1}{2}y)^2}{z+y+x+\bar n_A}-(D+\D+c(z+y+x+\bar n_A))x
\end{pmatrix},
\end{align}
which has a critical point at the origin. The Jacobian matrix of $Y$ at the fixed point $(0,0,0)$ is given by
\begin{align}
DY((0,0,0))=
\begin{pmatrix}
-(f-D) & -(f-D) & -(2f-D)\\
0 & 0 & 2f\\
0 & 0 & -(f+\D)\\
\end{pmatrix},
\end{align}
which has the eigenvalues $\lambda_1=-(f-D), \lambda_2=0$ and $\lambda_3=-(f+\D)$ with corresponding eigenvectors 
\begin{align}
EV_1=
\begin{pmatrix}
1 \\
0\\
0 
\end{pmatrix},\quad
EV_2=
\begin{pmatrix}
-1 \\
1\\
0 
\end{pmatrix},\quad
\text{ and }\quad
EV_3=
\begin{pmatrix}
\frac{fD+\D(2f-D)}{(f+\D)(D+\D)}\\
-\frac{2f}{f+\D}\\
1
\end{pmatrix}.
\end{align}
We can find a transformation
\begin{align}
T=
\begin{pmatrix}
1 & 1 & \frac D{D+\D}\\
0 & 1 & \frac{2f}{f+\D}\\
0 & 0 & 1
\end{pmatrix},
\end{align}
that transforms $DY((0,0,0))$ into a block matrix. The change of variables 
\begin{align}
\begin{pmatrix}
\tilde z \\
\tilde y\\
\tilde x 
\end{pmatrix}=T
\begin{pmatrix}
z\\
y\\
x
\end{pmatrix},
\end{align}
casts the system into the form
\be
\dot{\tilde z}=\dot z+\dot y+\frac D{D+\D}\dot x,\qquad
\dot{\tilde y}=\dot y+\frac {2f}{f+\D}\dot x,\qquad
\dot{\tilde x}=\dot x.
\ee
Let $h(\tilde y)$ be the local center manifold. We approximate $h(\tilde y)=\tbinom{h_1}{h_2}\tilde y^2+\OO(\tilde y^3)$ and substitute the series expansions into the center manifold equation \eqref{cme}
\begin{align}
\binom{2h_1\tilde y}{2h_2\tilde y}\dot{\tilde \tilde y}(h_1\tilde y^2,\tilde y,h_2\tilde y^2)=\binom{\dot{\tilde z}(h_1\tilde y^2,\tilde y,h_2\tilde y^2)}{\dot{\tilde x}(h_1\tilde y^2,\tilde y,h_2\tilde y^2)}.
\end{align}
To determine $h_1$ and $h_2$ we compare the coefficients of the same powers of $\tilde y$. We first consider $\dot{\tilde y}(h_1\tilde y^2,\tilde y,h_2\tilde y^2)$ and get that the coefficient of $\tilde y$ is zero. The first coefficients which are not zero at the right hand side are the ones of $\tilde y^2$. Thus we have to compare these with the coefficient of $\tilde y^2$ on the left hand side which is zero as mentioned above. Hence, the coefficient of $\dot{\tilde x}(h_1\tilde y^2,\tilde y,h_2\tilde y^2)$ of $\tilde y^2$, which is $-h_2(f+\D)+\frac f{4\bar n_A}$, equals zero and we get $h_2=\frac{f}{4\bar n_A(f+\D)}$. From the coefficient of $\dot{\tilde z}(h_1\tilde y^2,\tilde y,h_2\tilde y^2)$ of $\tilde y^2$, which is $Dh_1-\frac{f\D(2D+\D)}{4\bar n_A(f+\D)(D+\D)}$, we get $h_1=\frac{f\D(2D+\D)}{4\bar n_AD(f+\D)(D+\D)}$. Thus the local center manifold is given by 
\begin{align}
h(\tilde y)=
 \begin{pmatrix}
  \sfrac{f\D^2}{4\bar n_AD(f+\D)(D+\D)}\\
  \sfrac{f}{4\bar n_A(f+\D)}
 \end{pmatrix}
\tilde y^2+\OO(\tilde y^3).
\end{align}
Substitution of this result into $\dot{\tilde y}$ yields 
the flow on the local center manifold
\begin{align}
\dot{\tilde y}=-\frac{f\D}{2\bar n_A(f+\D)}\tilde y^2+\OO(\tilde y^3).
\end{align}
We can bound the solution of this equation with initial condition $y(0)=\e$ by
\begin{align}
\frac{2\bar n_A(f+\D)}{(f\D+\varrho) t+\sfrac{2\bar n_A(f+\D)}{\e}}\leq\tilde y(t)\leq\frac{2\bar n_A(f+\D)}{(f\D-\varrho) t+\sfrac{2\bar n_A(f+\D)}{\e}},
\end{align}
for $\varrho<f\D$.
Thus we see that $\mathfrak{n}_{AA}$ is a stable fixed point and $n_{aA}(t)$ approaches this fixed point like a function $\frac 1{t}$.
\end{proof}
In contrast  to the model of \citen{CMM13}, we see in the proof that $n_{aA}(t)$ does not dies out exponentially fast. Instead it decays like a function $f(t)=\frac{1}{t}$, as soon as $n_{aA}+n_{AA}\approx\bar n_A$ and thus survives a longer time in the population (cf. Figure \ref{fig1} and \ref{fig5}). Up to now, we know how the deterministic system evolves near its fixed points. Namely: if we start the process in a neighbourhood of the unstable fixed point $\mathfrak n_{aa}$, it will leave this neighbourhood in finite time. Whereas, if the process is in a neighbourhood of the stable fixed point $\mathfrak n_{AA}$, it converges to this point $\mathfrak n_{AA}$, but the convergence is slower than in the model of \citen{CMM13}. \\
We now turn to the analysis of the behaviour between these points. \\[0,3cm]
%-----------------------
\paragraph{Behaviour between the Fixed Points: }
We show that the deterministic system \eqref{det} moves from a neighbourhood of the unstable fixed point $\mathfrak n_{aa}$ to a neighbourhood of the stable fixed point $\mathfrak n_{AA}$ (Corollary \ref{fp}).
\begin{proof}[Proof of Corollary \ref{fp}]
The proof is similar to the proof of Theorem C.2 in \citen{CMM13}. It only differs in the last step. 
The general unperturbed vector field $X_0$ in the case of neutrality between the $a$- and $A$-alleles is given by
\begin{equation}
\label{undet}
X_0=
\begin{pmatrix}
f\frac{(x+\frac{1}{2}y)^2}{x+y+z}-(D+c(x+y+z))x\\
2f\frac{(x+\frac{1}{2}y)(z+\frac{1}{2}y)}{x+y+z}-(D+c(x+y+z))y\\
f\frac{(z+\frac{1}{2}y)^2}{x+y+z}-(D+c(x+y+z))z
\end{pmatrix}.
\end{equation}
The content of Theorem B.1 in \citen{CMM13} is that $X_0$ \eqref{undet} has a line of fixed points given by,
\begin{equation}
\Gamma_0(v)=
\begin{pmatrix}
\frac{(v-\bar n_A)^2}{4\bar n_A}\\
-\frac{v^2-\bar n_A^2}{2\bar n_A}\\
\frac{(v+\bar n_A)^2}{4\bar n_A}
\end{pmatrix}, \quad v\in [-\bar n_A,\bar n_A],
\end{equation}
and that 
the differential of the vector field $X_0$ at each point of the curve $\Gamma_0$, $DX_0(\Gamma_0(v))$, has the three eigenvectors
\be
e_1(v)=%\Gamma_0(v)=
\begin{pmatrix}
\frac{(v-\bar n_A)^2}{4\bar n_A}\\
-\frac{v^2-\bar n_A^2}{2\bar n_A}\\
\frac{(v+\bar n_A)^2}{4\bar n_A}
\end{pmatrix},\hspace{2mm}
e_2(v)=%\frac{d\Gamma_0(v)}{dv}=
\begin{pmatrix}
\frac{v-\bar n_A}{2\bar n_A}\\
-\frac{v}{\bar n_A}\\
\frac{v+\bar n_A}{2\bar n_A}
\end{pmatrix},\hspace{2mm}
e_3(v)=%\frac{d^2\Gamma_0(v)}{dv^2}=\frac{1}{2\bar n_A}
\begin{pmatrix}
1\\
-2\\
1
\end{pmatrix},
\ee
with respective eigenvalues $-(f-D)<0$, $0$ and $-f<0$.  $DX_0(\Gamma_0(v))^t$ has the three eigenvalues, $-f+D, 0$, and $-f$, with corresponding eigenvectors
\begin{align}
\b_1(v)&=\frac{1}{\bar n_A}
\begin{pmatrix}
1\\1\\1
\end{pmatrix},\hspace{2mm}
\b_2(v)=
\begin{pmatrix}
-\frac{v+\bar n_A}{\bar n_A}\\
-\frac{v}{\bar n_A}\\
-\frac{v-\bar n_A}{\bar n_A}
\end{pmatrix},\hspace{2mm}
\b_3(v)=
\begin{pmatrix}
\frac{(v+\bar n_A)^2}{2\bar n_A}\\
\frac{v^2-\bar n_A^2}{2\bar n_A}\\
\frac{(v-\bar n_A)^2}{2\bar n_A}
\end{pmatrix},
\end{align}
which satisfy, for any $i,j\in\{1,2,3\}$ and any $v$,
\begin{align}
\label{EVb}
\langle\b_i(v),e_j(v)\rangle=\d_{i,j}.
\end{align}
Next we analyse the asymptotics of the flow associated to the perturbed vector field, $X_\D$ \eqref{XD}, 
as $t\rightarrow\infty$. 
From Theorem B.1 in \citen{CMM13} we know that the curve of  fixed points,  $\G_0$, is transversally hyperbolic and 
invariant for the vector field $X_0$. Thus, Theorem 4.1 in \citen{H77} implies that, for small enough $\D$,  there is an 
attractive curve, $\G_\D$ that is  invariant under $X_\D$.
 Moreover, $\G_\D$ is regular and converges to $\G_0$, as  $\D\rightarrow0$. Hence,  there is a small  tubular 
 neighbourhood, $\VV$,  of $\G_0$ such that $\G_\D$ is contained in $\VV$ and attracts all orbits with initial conditions in
  $\VV$ (cf. Figure \eqref{fig3}).
 
We want to study the flow associated to the vector field $X_\D$. From the remarks above we know that the curve
 $\G_\D$ is attractive for this flow. Thus, it suffices to analyse the flow on the curve. Precisely, we want to show that the 
 vector field does not vanish on $\G_\D$ except for the two fixed points, $\mathfrak n_{aa}$ and $\mathfrak n_{AA}$. 
 Thus  the orbit of any initial condition on $\G_\D$ converges to one of the two fixed points. 
 It is easier to look for the fixed points in the tube $\VV$, which is equivalent to looking for fixed points on $\G_\D$ 
 because of the attractiveness
 of this curve. 
Since Hirsch gives us only a local statement, we consider $\VV$ in local frames. A point $(x,y,z)\in\VV$ is represented by the parametrisation
\begin{align}
M(v,r,s)=\G_0(v)+re_1(v)+se_3(v)=(1+r)\G_0(v)+s\frac{d^2\G_0(v)}{dv^2},
\end{align}
with $v\in[-\bar n_A-\d,\bar n_A+\d]$ and $r,s\in[-\d,\d]$, with $\d>0$ chosen small enough. The determinant of the Jacobian matrix of the transformation $(v,r,s)\mapsto(x,y,z)=M(v,r,s)$ is $-\frac{r+1}{2}$ and thus invertible and does not vanish if $0<\d<1$. Moreover, it is a diffeomorphism which maps $[-\bar n_A-\d,\bar n_A+\d]\times [-\d,\d]^2$ to a closed neighbourhood of $\VV$.\\
For finding the zero points in $\VV$, we use linear combinations of the left eigenvectors $\b_i, i\in\{1,2,3\}$, with the 
perturbed vector field $X_\D$. First we look for zeros of  the two linear combinations of $X_{\D}$
 with the eigenvectors $\b_1$ and $\b_3$ 
 which spans the stable affine subspace. By the implicit function Theorem, we obtain a curve which contains all possible 
 zeros in $\VV$. Then we consider the last linear combination of $X_{\D}$ with $\b_2$ and zeros of this linear combination 
 on the curve above.

\begin{proposition}[Proposition B.2 in \citen{CMM13}]
\label{curve}
For any $\d>0$ small enough, there is a number $\D_0=\D_0(\delta)$ such that, for any $\D\in[-\D_0,\D_0]$, there is a smooth curve $\ZZ_\D=(r_\D(v), s_\D(v))\subset\R^2$, depending smoothly on $\D$ and converging to $0$ when $\D$ tends to zero such that, for any $v\in[-\bar n_A-\delta,\bar n_A+\delta]$, we have
\begin{equation}
\langle\b_1(v), X_\D(M(v, r_\D(v), s_\D(v)))\rangle=\langle\b_3(v), X_\D(M(v, r_\D(v), s_\D(v)))\rangle=0.
\end{equation}
Moreover, if a point $(v, r, s)$ with $v\in[-\bar n_A-\delta, \bar n_A+\delta], r$ and $s$ small enough is such that
\begin{equation}
\langle\b_1(v), X_\D(M(v, r, s))\rangle=\langle\b_3(v), X_\D(M(v, r, s))\rangle=0,
\end{equation}
then $(r, s)=(r_\D(v), s_\D(v))$.
\end{proposition}
Next, we look the points of the resulting curve (obtained from Proposition \ref{curve}) where the third linear combination of
 the components vanishes. Since 
\be
\langle\b_2(v),X(0,M(v,r_0(v),s_0(v)))\rangle
=\langle\b_2(v),X(0,M(v,0,0))\rangle
=\langle\b_2(v),X(0,\G_0(v))\rangle
=0,
\ee
this function vanishes for $\D=0$ in $v$ and we apply the Malgrange preparation Theorem \cite {GG73}, 
which provides a representation for the linear combination near $\D=0$:
\begin{align}
\label{MPT}
\langle\b_2(v),X(\D,M(v,r_\D(v),s_\D(v)))\rangle=\D^2h(\D,v)+\D g(v),
\end{align}
where $h,g$ are two smooth functions.
To show that the third linear combination indeed vanishes only in small neighbourhoods of the points $\pm\bar n_A$, \citen{CMM13} use a representation for the function $g$ which is independent of the perturbation $\D$. \begin{lemma}[Lemma B.3 in \citen{CMM13}]
The function $g$ in \eqref{MPT} is given by
\begin{align}
g(v)=\langle\b_2(v),\partial_\D X_0(\G_0(v))\rangle.
\end{align}
\end{lemma}
Then they ensure that these function has only two zeros, which implies, because of the representation \eqref{MPT},
 for $|\D|\neq0$ small enough, that the perturbed vector field $X_\D$ has only the two  known fixed points 
 $\mathfrak n_{aa}$ and $\mathfrak n_{AA}$.
\begin{theorem}[Theorem B.4 in \citen{CMM13}]
\label{B4}
Assume the function
\begin{align}
\label{g}
g(v)=\langle\b_2(v),\partial_\D X_0(\G_0(v))\rangle,
\end{align}
satisfies $\frac{dg}{dv}(\pm\bar n_A)\neq0$ and does not vanish in $(-\bar n_A,\bar n_A)$. Then, for $|\D|\neq0$ small enough, the vector field $X_\D$ has only two zeros in a tubular neighbourhood of $\G_0$. 
These zeros are $(n_{aa}(\D), 0, 0)$ and $(0, 0, n_{AA}(\D))$,
 with $n_{aa}(\D)$ and $n_{AA}(\D)$ regular near $\D=0$ and $n_{aa}(0)=n_{AA}(0)=\bar n_A$.
\end{theorem}
While the hypothesis $\frac{dg}{dv}(\pm\bar n_A)\neq0$  does not hold here, the conclusion of 
Theorem \ref{B4} remains true. Namely, we have 
from \eqref{XD} that 
\begin{equation}
\label{easycal1}
\partial_\D X_\D(x,y,z)=
\begin{pmatrix}
-x\\0\\0
\end{pmatrix},
\end{equation}
and thus
\begin{align}
\label{easycal2}
g(v)=\langle\b_2(v),\partial_\Delta X_0(\G_0(v))\rangle=
\left\langle
\begin{pmatrix}
-\frac{v+\bar n_A}{\bar n_A}\\
-\frac{v}{\bar n_A}\\
-\frac{v-\bar n_A}{\bar n_A}
\end{pmatrix},
\begin{pmatrix}
-\frac{(v-\bar n_A)^2}{4\bar n_A}\\0\\0
\end{pmatrix}
\right\rangle
=\frac{(v+\bar n_A)(v-\bar n_A)^2}{4\bar n_A^2}.
\end{align}
Obviously, $g(\pm\bar n_A)=0$ and $g$  has no other zeros, in particular,  it does not vanish in $(-\bar n_A,\bar n_A)$.  
Hence, it follows from the representation \eqv(MPT) that there is a $\delta>0$ such that, for $\Delta$ small enough,
\\  $\langle\beta_2(v), X_\D(M(v,r_\D(v),s_\D(v)))\rangle$ has only two zeros in $v\in[-\delta-\bar n_A,\bar n_A+\delta]$. 
From Corollary \ref{fp} we get the existence of these two fixed points, which have to be the points $\mathfrak n_{aa}$ and
 $\mathfrak n_{AA}$.
Finally, we need the following lemma, which is analogous to Theorem C.1 in \citen{CMM13}.
\begin{lemma}
\hspace{-0pt}
\begin{itemize}
\item[(a)] The local stable manifold of the unstable fixed point $\mathfrak n_{aa}=(\bar n_a, 0, 0)$ intersects the closed 
positive quadrant only along the line $y=z=0$.
\item[(b)] The local unstable manifold is contained in the curve $\G_\D$.
\end{itemize}
\end{lemma}
\begin{proof}
We start by proving (a). From Theorem \ref{dettheo} we get the hyperbolicity, thus we can apply Theorem 4.1 in \citen{H77}. The Jacobian matrix $DX_\D((\bar n_a, 0, 0))$ (cf. \eqref{DXaa}) has the three eigenvalues $\lambda_1=-(f-D-\D), \lambda_2=\D$ and $\lambda_3=-(f-\D)$ with corresponding eigenvectors
\begin{align}
EV(\lambda_1)&=
\begin{pmatrix}
1\\0\\0
\end{pmatrix}
=e_1(-\bar n_A)\\
EV(\lambda_2)&=
\begin{pmatrix}
-1\\1\\0
\end{pmatrix}
+\frac{\D}{f-D}
\begin{pmatrix}
1\\0\\0
\end{pmatrix}
=e_2(-\bar n_A)+\OO(\D)\\
EV(\lambda_3)&=\frac{1}{2\bar n_A}
\begin{pmatrix}
1\\-2\\1
\end{pmatrix}
+\frac{\D}{2\bar n_A(f-D)}
\begin{pmatrix}
1\\0\\0
\end{pmatrix}
=e_3(-\bar n_A)+\OO(\D).
\end{align}
Let $\EE_{aa}^s(\D)$ the two dimensional affine stable subspace spanned by the eigenvectors 
$EV(\lambda_1)$ and $EV(\lambda_3)$ 
with origin in $\mathfrak n_{aa}$:
\begin{equation}
\label{E}
\EE_{aa}^s(\D)=\{x\in\R^3|x=(\bar n_a, 0, 0)^t+sEV(\lambda_1)+tEV(\lambda_3), \forall s,t\in\R\}.
\end{equation}
 Again, by the lamination and permanence condition in Theorem 4.1 in \citen{H77}, 
 we get the existence of a stable manifold 
 $W_{aa}^{s,loc}$ and an unstable manifold $W_{aa}^{u,loc}$ of the fixed point $\mathfrak n_{aa}$. 
 The local stable manifold $W_{aa}^{s,loc}$ is a piece of regular manifold tangent in $\mathfrak n_{aa}$ to the subspace 
 $\EE_{aa}^s(\D)$. We see that the $x$-axis is invariant for $X_\D$ and is contained in $W_{aa}^{s,loc}$. From \eqref{E}, 
 we get that $\EE_{aa}^s(\D)$ intersects the closed positive quadrant only along the line $y=z=0$. Hence, 
 the same is true for $W_{aa}^{s,loc}$ since it is a piece of the subspace $\EE_{aa}^s(\D)$. This shows (a).
To show (b), we show that the local unstable manifold $W_{aa}^{u,loc}$ is contained in the closed positive quadrant. 
This follows because $W_{aa}^{u,loc}$ is tangent to the linear unstable direction in $EV(\lambda_2)$ in $\mathfrak n_{aa}$, 
which points into the positive quadrant.
From Theorem 4.1 in \citen{H77} we get that the invariant curve, $\G_\D$, is unique, thus $W_{aa}^{u,loc}\subset\G_\D$
and (b) follows by the invariance of the positive quadrant under the flow.
\end{proof}
The preceding steps conclude the proof of Corollary \ref{fp}. 
\end{proof}
%-----------------------------------------------------------------------------------
\subsection{Proof of the main theorems in Section \ref{sectionmain}}
\label{proofmain}
We carry out the proofs of the main theorems (Section \ref{sectionmain}) in full detail.\\
The mutant process $A(t)=2n_{AA}(t)+n_{aA}(t)$ jumps up (resp. down) by rate $b_A$ (resp. $d_A$) given by:
\begin{align}
b_A&=\frac{2fK}{\S(t)}\left(\left(n_{AA}(t)+\sfrac{1}{2}n_{aA}(t)\right)^2+\left(n_{AA}(t)+\sfrac{1}{2}n_{aA}(t)\right)\left(n_{aa}(t)+\sfrac{1}{2}n_{aA}(t)\right)\right)\nonumber\\
&=fK(2n_{AA}(t)+n_{aA}(t)) =A(t)Kf,\\
d_A&=2n_{AA}(t)K(D+c\S(t))+n_{aA}(t)K(D+c\S(t))=A(t)K(D+c\S(t)).
\end{align}
%----------------------------------------------------------------------------------
\subsubsection{Phase 1: Fixation of the mutant population}
Recall the stopping times \eqref{stoppaA+} and \eqref{stoppaA-}, when the mutant population $A(t)$ increased to a $\d$-level and its stopping time of extinction. We show 
Theorem \ref{fix}:
\begin{proof}[Proof of Theorem \ref{fix}]
We start the population process with a monomorphic $aa$-population which stays in a $\d/2 K$-neighbourhood of its equilibrium $\bar n_aK$ and one mutant with genotype $aA$, i.e. $\t_0=0$. Because of \eqref{mu}, there will be no further mutation in the process.\\
Proposition \ref{propD} states that if the resident population $n_{aa}$ is in a $\d/2$-neighbourhood of its equilibrium $\bar n_a$ then $n_{aa}$ will stay in a $\d$-neighbourhood an exponentially long time as long as the mutant population is less than $\d$. Hence we get that, as long as the mutant population is smaller than $\d$, the time the process $n_{aa}(t)$ needs to exit from its domain $\bar n_a$ is bigger than $\exp(VK)$ with probability converging to 1, for some $V>0$ (cf. \citen{C06}) and the dynamics of the mutant population are negligible for $n_{aa}(t)$.\\
With this knowledge we analyse the fate of the mutants for $t<\tau_\d^{mut}\land\tau_0^{mut}\land e^{VK}$. We use the comparison results of birth and death processes (Theorem 2 in \citen{C06}) to estimate the mutant process from below and above. 
We denote by $\preccurlyeq$ the following stochastic dominant relation: if $\bf{P_1}$ and $\bf P_2$ are the laws of  $\R$-valued processes, we will write $\bf{P_1}\preccurlyeq\bf{P_2}$ if we can construct on the same probability space $(\Omega,\FF,\bf {P})$ two processes $X^1$ and $X^2$ such that the law of the processes is $\LL(X^i)=\bf{P_i}$, $i\in\{1,2\}$ and for all $t>0, \o\in\Omega: X_t^1(\o)\leq X_t^2(\o)$.\\
First we construct a process $A_l(t)\preccurlyeq A(t)$ which is the minorising process of the mutant process. This process has the birth and death rates:
\begin{equation}
b_l(t)=A_l(t)Kf,\qquad
d_l(t)%  =A_l(t)K[D+c(\bar n_a+\d_0)+c\d]\nonumber\\
=A_l(t)K[D+c(\bar n_a+(M+1)\d)].
\end{equation}
$A(t)\preccurlyeq A_u(t)$ is the majorising process with rates:
\be
b_u(t)=A_u(t)fK,\qquad
d_u(t)=A_u(t)K[D+c(\bar n_a-M\d)].
\ee
We define the stopping times 
\begin{align}
T_{n/K}^{l}&\equiv\inf\left\{t\geq0:A_{l}(t)=\frac{n}{K}\right\},\qquad T_{n/K}^{u}\equiv\inf\left\{t\geq0:A_{u}(t)=\frac{n}{K}\right\}\\
\T_a&\equiv\inf\left\{t\geq0: |n_{aa}(t)-\bar n_a|>\d\right\},
\end{align}
which are the first times that the processes $A_l$, resp. $A_u$ reach the level $\frac{n}{K}$ and the 
exit time of $n_{aa}(t)$ from the domain $[\bar n_a-\d,\bar n_a+\d]$.

Note that both processes $A_l(t)$ and $A_u(t)$ are super-critical. 
In the following we use the results for super-critical branching processes proven in \citen{C06}:
\begin{lemma}[Theorem 4 (b) in \citen{C06}]
\label{theorem4b}
Let $b,d>0$. For any $K\geq1$ and any $z\in\N/K$, let $\P^K_z$ the law of the $\N/K$-valued Markov linear birth and death process 
$\left(\o_t,t\geq0\right)$ with birth and death rates $b$ and $d$ and initial state $z$. Define, for any $\rho\in\R$, on $\mathbb{D}(\R_+,\R)$, the stopping time
\begin{align}
T_\rho=\inf\{t\geq0: \o_t=\rho\}.
\end{align}
Let $(t_K)_{K\geq1}$ be a sequence of positive numbers such that $\ln K\ll t_K$.
If $b>d$, for any $\e>0$,
\begin{align}
&(a)\quad \lim_{K\rightarrow\infty}\P_{\frac1K}^K(T_0\leq t_K\land T_{\lceil\e K\rceil/K})=\frac db,\\
&(b)\quad \lim_{K\rightarrow\infty}\P_{\frac1K}^K( T_{\lceil\e K\rceil/K}\leq t_K)=1-\frac db.
\end{align}
\end{lemma}
With respect to Theorem \ref{main2}, we prove the theorem for arbitrary $\mu_K$. From\eqref{muttime} we know that the next mutation occurs with high probability not before a time $\frac\rho{K\mu_K}$. Then
\begin{align}
\label{up}
\P\hspace{-2pt}_{\frac{1}{K}}\hspace{-3pt}\left[T_{\lceil\d K\rceil/K}^l<T_0^l\land\sfrac{\r}{K\mu_K}\land\T_a\right]\leq\P\hspace{-2pt}_{\frac{1}{K}}\hspace{-3pt}\left[\t_{\d }^{mut}<\t_0^{mut}\land\sfrac{\r}{K\mu_K}\land\T_a\right]
\leq\P\hspace{-2pt}_{\frac{1}{K}}\hspace{-3pt}\left[T_{\lceil\d K\rceil/K}^u<T_0^u\land\sfrac{\r}{K\mu_K}\land\T_a\right].
\end{align}
Using Proposition \ref{propD}, there exists $V>0$ such that, for sufficiently large $K$,
\begin{align}
\P_{\frac1K}[\t_1\land e^{VK}<\T_a]\geq1-\d.
\end{align}
We can estimate, for $K$ large enough,
\begin{align}
\P\hspace{-2pt}_{\frac{1}{K}}\hspace{-3pt}\left[T_{\lceil\d K\rceil/K}^l<T_0^l\land\sfrac{\r}{K\mu_K}\land\T_a\right]
&\geq\P\hspace{-2pt}_{\frac{1}{K}}\hspace{-3pt}\left[T_{\lceil\d K\rceil/K}^l<T_0^l\land\sfrac{\r}{K\mu_K}\land e^{VK}\right]-\d\nonumber\\
&=\P\hspace{-2pt}_{\frac{1}{K}}\hspace{-3pt}\left[T_{\lceil\d K\rceil/K}^l<T_0^l\land\sfrac{\r}{K\mu_K}\right]-\d\nonumber\\
&\geq\P\hspace{-2pt}_{\frac{1}{K}}\hspace{-3pt}\left[\left\{T_{\lceil\d K\rceil/K}^l<T_0^l\land\sfrac{\r}{K\mu_K}\right\}\cap\left\{T_0^l\leq\sfrac{\r}{K\mu_K}\right\}\right]-\d\nonumber\\
&=\P\hspace{-2pt}_{\frac{1}{K}}\hspace{-3pt}\left[\left\{T_{\lceil\d K\rceil/K}^l<T_0^l\right\}\cap\left\{T_0^l\leq\sfrac{\r}{K\mu_K}\right\}\right]-\d\nonumber\\
&\geq \P\hspace{-2pt}_{\frac{1}{K}}\hspace{-3pt}\left[T_{\lceil\d K\rceil/K}^l<T_0^l\right] -\P_{\frac 1K}\left[T_0^l>\sfrac{\r}{K\mu_K}\right]-\d\nonumber\\
\end{align}
The extinction time for a binary branching process when $b\neq d$ is given by (cf. page 109 in \citen{AN11}) 
\begin{align}
\P_n(T_0\leq t)=\left(\frac{d(1-e^{-(b-d)t})}{b-de^{-(b-d)t}}\right)^n,
\end{align}
for any $t\geq0$ and $n\in\N$. Under our condition on $\mu_K$ \eqref{mu} in Theorem \ref{main2}, this implies that 
\be
\lim_{K\rightarrow\infty}\P\hspace{-2pt}_{\frac{1}{K}}\left[T_0^l\leq\sfrac{\r}{K\mu_K}\right]=0.
\ee
 Together with Lemma \ref{theorem4b}, we get
\begin{align} \label{immer.1}
\lim_{K\rightarrow\infty}\P\hspace{-2pt}_{\frac{1}{K}}\hspace{-3pt}\left[\t_{\d }^{mut}<\t_0^{mut}\land\sfrac{\r}{K\mu_K}\land\T_a\right]
&\geq 1-\frac{d_l}{b_l}-\d
=\frac{\D}{f}-\left(\frac{c(M+1)}{f}+1\right)\d.
\end{align}
If we next consider the upper bound of \eqref{up}, we see that
\begin{align}
P\hspace{-2pt}_{\frac{1}{K}}\hspace{-3pt}\left[T_{\lceil\d K\rceil/K}^u<T_0^u\land\sfrac{\r}{K\mu_K}\land\T_a\right]\leq\P\hspace{-2pt}_{\frac{1}{K}}\hspace{-3pt}\left[T_{\lceil\d K\rceil/K}^u<T_0^u\land\sfrac{\r}{K\mu_K}\right].
\end{align}
Similar as in \eqref{immer.1}
\be
\lim_{K\rightarrow\infty}\P_{\frac{1}{K}}\hspace{-3pt}\left[\t_{\d }^{mut}<\t_0^{mut}\land\sfrac{\r}{K\mu_K}\land\T_a\right]
=\frac{\D}{f}+\frac{cM\d}{f}.
\ee
Now, we let $\d\rightarrow0$ and get the desired result.
\end{proof}
%--------------------------------------------
\subsubsection{Phase 2: Invasion of the mutant}
If the mutant invades in the resident $aa$-population then the first phase ends with a macroscopic mutant population. Especially, we know that the $aA$-population is of order $\d$ due to its advantage in recombination in contrast to the $AA$-population. Applying the Large Population Approximation (Theorem \ref{LPA}, \citen{F_MA}) with this initial condition, we get that the behaviour of the process is close to the solution of the deterministic system \eqref{det}, when $K$ tends to infinity. We approximate the population process by the solution of the dynamical system \eqref{det}. Result \ref{inva} in Proposition \ref{det} is known (see \citen{CM11} or \citen{CMM13}). We use this result only until the $aA$-population decreases to an $\e$-level. %------------------------------------------------
\subsubsection{Phase 3: Survival of the recessive $a$-Allele}
This phase starts as soon as the $aA$-population hits the $\e$ value. At this time we restart the population-process, that means we set the time to zero. 
The analysis of the deterministic dynamical system up to this point shows that we get the following initial conditions:
\begin{align}
&n_{aA}(0)=\e,\\
&n_{aa}(0)\leq\frac {f}{4\bar n_A(f+\D)}\e^2+M_{aa}\e^{2+\a},\\
&|n_{AA}(0)-(\bar n_A-\e)|\leq M_{AA}\e^2,\\
&\left|\S(0)-\left(\bar n_A-\sfrac \Delta{c\bar n_A}\g\e^2\right)\right|\leq M_\S\e^{2+\a},
\end{align}
where $M_i, i\in\{aa,AA,\S\}$ are constants.\\
In the following stopping times denoted by $\t$ are always stopping times on rescaled processes, whereas stopping times denoted by $T$ are the stopping times of the corresponding non-rescaled processes.\\
We transform the birth and death rates of the processes
$N_{aa}(t), N_{aA}(t), N_{AA}(t)$ and the sum process $\S(t)K$
in such a way that we can consider them as the birth and death rates of  linear birth-death-immigration processes:
\begin{align}
\label{raten1}
b_{\S}(t)&=f\S(t)K,\nonumber\\
d_{\S}(t)&=D\S(t)K+c\S^2(t)K+\D N_{aa}(t),\\[6pt]
\label{raten2}
b_{aa}(t)&=fN_{aa}(t)\left(1-\frac{n_{AA}(t)}{\S(t)}\right)+\frac{fK}{4\S(t)}n^2_{aA}(t),\nonumber\\
d_{aa}(t)&=N_{aa}(t)(D+\D+c\S(t)),\\[6pt]
\label{raten3}
b_{aA}(t)&=fN_{aA}(t)\left(1-\frac{n_{aA}(t)}{2\S(t)}\right)+2fN_{aa}(t)\frac{n_{AA}(t)}{\S(t)},\nonumber\\
d_{aA}(t)&=N_{aA}(t)(D+c\S(t)),\\[6pt]
\label{raten4}
b_{AA}(t)&=fN_{AA}(t)\left(1-\frac{n_{aa}(t)}{\S(t)}\right)+\frac{fK}{4\S(t)}n^2_{aA}(t),\nonumber\\
d_{AA}(t)&=N_{AA}(t)(D+c\S(t)).
\end{align}
We proceed as described in the outline.
Recall the settings for the steps \eqref{defx}, \eqref{stoppaA+},\eqref{stoppaA-} and \eqref{ubi}.
\\[6pt]
%-------------------------------------------
\paragraph{Step 1:} 
We prove the upper bound of the sum process (Proposition \ref{upSum}). For this we construct a process, the difference process, which records the drift from the sum process away from the upper bound $\bar n_AK$.
\begin{proof}[Proof of Proposition \eqref{upSum}]
The difference process $X_t^{u\S}$ between $\S(t)K$ and $\bar n_AK$ is a branching process with the same rates as $\S(t)K$. Set
	\bea
		X_t^{u\S}&=&\S(t)K-\bar n_AK,\nonumber\\
		T_0^{X,u\S}&\equiv&\inf\{t\geq0: X_t^{u\S}=0\},\nonumber\\
		T_{\a,M_\S}^{X,u\S}&\equiv&\inf\{t\geq0: X_t^{u\S}\geq3M_{\S}(x^{2i}\e^2)^{1+\a}K\}.
	\eea
This is a process in continuous time. For the following we need the discrete process $Y_n^{u\S}$ associated to $X_t^{u\S}$. To obtain this process we introduce a sequence of stopping times $\vartheta_n$ which records the times, when $X_t^{u\S}$ makes a jump. Formally, $\vartheta_n$ is the smallest time $t$ such that $X_t^{u\S}\neq X_{s}^{u\S}$, for all $\vartheta_{n-1}\leq s<t$. We set $X_{\vartheta_n}^{u\S}=Y_n^{u\S}$.
This discretisation has probability less than $\frac12$ to make an upward jump:
\begin{lemma}
	For  $\in \N$ such that $\vartheta_n\in \left[\t_{aA}^{i-}, \t_{aA}^{i+}\land\t_{aA}^{(i+1)-}\land e^{VK^{\a}}\right]$ and $1\leq k\leq 3M_\S (x^{2i}\e^2)^{1+\a}K$ there exists a constant $C_0>0$ such that
	\begin{equation}
		\P[Y_{n+1}^{u\S}=k+1|Y_n^{u\S}=k]\leq \frac{1}{2}-C_0kK^{-1}\equiv p_{\S}(k).
	\end{equation}
\end{lemma}
\begin{proof}
With the rates \eqref{raten1}, we get by some straightforward computations
	\begin{align}
		\P[Y_{n+1}^{u\S}=k+1|Y_n^{u\S}=k]&=\frac{b_\S}{b_\S+d_\S}=\frac{(\bar n_AK+k)f}{(\bar n_AK+k)[f+D+c(\bar n_A+kK^{-1})]+\D N_{aa}(t)} \nonumber\\
		&=\frac{1}{2}+\frac{\frac{1}{2}f-\frac{1}{2}D-\frac{1}{2}c\bar n_A-\frac{1}{2}ckK^{-1}-\frac12\frac{\D N_{aa}(t)}{\bar n_AK+k}}{f+D+c(\bar n_A+kK^{-1})+\frac{\D N_{aa}(t)}{\bar n_AK+k}}\nonumber\\
		&\leq\frac{1}{2}-C_0kK^{-1}.
	\end{align}
\end{proof}
To obtain a Markov chain we couple the process $Y_n^{u\S}$ to a process $Z_n^{u\S}$ via:
	\begin{align}
	&(1)\quad Z_0^{u\S}=Y_0^{u\S}\lor 0\\
	&(2)\quad \P[Z_{n+1}^{u\S}=k+1|Y_n^{u\S}<Z_n^{u\S}=k]=p_{\S}(k),\\
	&(3)\quad \P[Z_{n+1}^{u\S}=k-1|Y_n^{u\S}<Z_n^{u\S}=k]=1-p_{\S}(k),\\
	&(4)\quad \P[Z_{n+1}^{u\S}=k+1|Y_{n+1}^{u\S}=k+1, Y_n^{u\S}=Z_n^{u\S}=k]=1,\\
	&(5)\quad \P[Z_{n+1}^{u\S}=k+1|Y_{n+1}^{u\S}=k-1, Y_n^{u\S}=Z_n^{u\S}=k]=\frac{p_{\S}(k)-\P[Y_{n+1}^{u\S}=k+1|Y_n^{u\S}=k]}{1-\P[Y_{n+1}^{u\S}=k+1|Y_n^{u\S}=k]},\\
	&(6)\quad \P[Z_{n-1}^{u\S}=k+1|Y_{n+1}^{u\S}=k-1, Y_n^{u\S}=Z_n^{u\S}=k]=1-\frac{p_{\S}(k)-\P[Y_{n+1}^{u\S}=k+1|Y_n^{u\S}=k]}{1-\P[Y_{n+1}^{u\S}=k+1|Y_n^{u\S}=k]}.
	\end{align}
Observe that by construction $Z_n^{u\S}\succcurlyeq Y_n^{u\S}$, a.s. and that $\P[Z_{n+1}^{u\S}=1|Z_{n}^{u\S}=0]=1$. The marginal distribution of $Z_n^{u\S}$ is the desired Markov chain with transition probabilities 
	\begin{align}
		\P[Z_{n+1}^{u\S}=k+1|Z_n^{u\S}=k]
		=&p_{\S}(k)\P[Y_n^{u\S}<Z_n^{u\S}|Z_n^{u\S}=k]+\P[Y_{n+1}^{u\S}=k+1|Y_n^{u\S}=k]\P[Y_n^{u\S}=Z_n^{u\S}|Z_n^{u\S}=k]\nonumber\\
		&+\sfrac{p_{\S}(k)-\P[Y_{n+1}^{u\S}=k+1|Y_n^{u\S}=k]}{1-\P[Y_{n+1}^{u\S}=k+1|Y_n^{u\S}=k]}\P[Y_n^{u\S}=Z_n^{u\S}|Z_n^{u\S}=k](1-\P[Y_{n+1}^{u\S}=k+1|Y_n^{u\S}=k])\nonumber\\
		=&p_{\S}(k)(\P[Y_n^{u\S}<Z_n^{u\S}|Z_n^{u\S}=k]+\P[Y_n^{u\S}=Z_n^{u\S}]|Z_n^{u\S}=k)
		=p_{\S}(k),\\
		\P\left(Z_{n+1}^{u\S}=k-1|Z_n^{u\S}=k\right)
		=&1-p_{\S}(k),
	\end{align}
and invariant measure
	\begin{align}
		\mu(n)=\frac{\prod_{k=1}^{n-1}\left(\frac{1}{2}-C_0kK^{-1}\right)}{\prod_{k=1}^{n}\left(\frac{1}{2}+C_0kK^{-1}\right)}, \quad\text{for }n\geq2,
	\end{align}
with $\mu(0)=1$ and $\mu(1)=\frac{1}{\frac{1}{2}+C_0K^{-1}}$.
We want to calculate the probability that the Markov chain $Z_n^{u\S}$, starting in a point $zK$, reaches first $3M_{\S}(x^{2i}\e^2)^{1+\a}K$ before going to zero, which is the equilibrium potential of a one dimensional chain \citen{B06}.
Using  (\citen{BH15}, Chapter 7.1.4, Eq. 7.1.57), we get (for $K$ large enough)
	\begin{align}
		\P_{zK}\left[T_{\a,M_\S}^{Z^{u\S}}<T_0^{Z^{u\S}}\right]&=\frac{\sum_{n=1}^{zK}\frac{1}{1-p_\S(n)}\frac{1}{\mu(n)}}{\sum_{n=1}^{3M_{\S}(x^{2i}\e^2)^{1+\a}K}\frac{1}{1-p_\S(n)}\frac{1}{\mu(n)}}\nonumber\\
		&=\frac{\sum_{n=1}^{zK}\prod_{k=1}^{n-1}\frac{1+2C_0kK^{-1}}{1-2C_0kK^{-1}}}{\sum_{n=1}^{3M_{\S}(x^{2i}\e^2)^{1+\a}K}\prod_{k=1}^{n-1}\frac{1+2C_0kK^{-1}}{1-2C_0kK^{-1}}}\nonumber\\
		&=\frac{\sum_{n=1}^{zK}\exp\left(\sum_{k=1}^{n-1}\ln\left(\frac{1+2C_0kK^{-1}}{1-2C_0kK^{-1}}\right)\right)}{\sum_{n=1}^{3M_{\S}(x^{2i}\e^2)^{1+\a}K}\exp\left(\sum_{k=1}^{n-1}\ln\left(\frac{1+2C_0kK^{-1}}{1-2C_0kK^{-1}}\right)\right)}\nonumber\\
		&\leq\frac{\sum_{n=1}^{zK}\exp\left(\sum_{k=1}^{n-1}4C_0kK^{-1}\right)}{\sum_{n=1}^{3M_{\S}(x^{2i}\e^2)^{1+\a}K}\exp\left(\sum_{k=1}^{n-1}4C_0kK^{-1}-\OO\left((kK^{-1})^2\right)\right)}\nonumber\\
		&\leq\frac{\sum_{n=1}^{zK}\exp\left(2C_0n^2K^{-1}-2C_0nK^{-1}\right)}{\sum_{n=1}^{3M_{\S}(x^{2i}\e^2)^{1+\a}K}\exp\left(2C_0n^2K^{-1}-2C_0nK^{-1}-\OO\left(n(nK^{-1})^2\right)\right)}
		\nonumber\\
		&\leq\frac{zK\exp\left(2C_0z^2K\right)}{\sum_{n=2M_{\S}(x^{2i}\e^2)^{1+\a}K}^{3M_{\S}(x^{2i}\e^2)^{1+\a}K}\exp\left(2C_0n^2K^{-1}-2C_0nK^{-1}-\OO\left(n(nK^{-1})^2\right)\right)}
\nonumber\\
		&\leq\frac{z}{M_{\S}(x^{2i}\e^2)^{1+\a}}\frac{\exp\left(2C_0z^2K\right)}{\exp\left(8C_0M_{\S}^2(x^{4i}\e^4)^{1+\a}K-4C_0M_{\S}(x^{2i}\e^2)^{1+\a}-\OO\left((x^{6i}\e^6)^{1+\a}K\right)\right)}
		\nonumber\\
		&\leq\exp\left(-2C_0K\left(2M_{\S}^2(x^{4i}\e^4)^{1+\a}-z^2\right)\right).
	\end{align}
%---
We denote by $R$ the number of times that the   process $Z_n^{u\S}$  returns to zero before   it reaches $3M_\S (x^{2i}\e^2)^{1+\a}K$.
Let  $R_z^k=\P_{zK}[R=k]$ be the probability that this number is $k$ when starting in $zK$.  
We define the times of the $n$-th returns to zero:
	\be
		T_0^1=\inf\{t>0: Z_t^{u\S}=0\},\qquad
		T_0^n=\inf\{t>T_0^{n-1}: Z_t^{u\S}=0\}.
	\ee
We than have
	\begin{align}
		R_z^k=\P_{zK}\left[T_0<T_{\a,M_{\S}}^{Z^{u\S}}\right](1-\P_0[T_{\a,M_{\S}}^{Z^{u\S}}<T_0])^{k-1}\P_0[T_{\a,M_{\S}}^{Z^{u\S}}<T_0]\leq(1-A)^{k-1}A,
	\end{align}
where 
\begin{align}
A\equiv \P_0[T_{\a,M_{\S}}^{Z^{u\S}}<T_0]
&=\sum_{i\geq1}p(0,i)\P_i[T_{\a,M_\S}^{Z^{u\S}}<T_0]=p(0,1)\P_1[T_{\a,M_\S}^{Z^{u\S}}<T_0]\nonumber\\
&\leq\exp\left(-\tilde C_0KM_{\S}^2(x^{4i}\e^4)^{1+\a}\right)\leq\exp\left(-\tilde C_0M_{\S}^2K^{3\a}\right), 
\end{align}
for some finite positive constant $\tilde C_0$.
 Then 
	\be
	\label{R<=N}
		\P[R\leq N]\leq\sum_{i=1}^NR_z^i\leq\sum_{i=1}^N(1-A)^{i-1}A
		=1-(1-A)^N.
	\ee
We choose, e.g.,  $N\sim \frac{1}{K^2A}$, so that $\P\left[R\leq N\right]=o\left(K^{-1}\right)$. Let $I_0\equiv T_0^1$ and $I_n\equiv T_0^n-T_0^{n-1}$ the time the process needs for return to zero. The $(I_j)_{j\in\N}$ are i.i.d. random variables and it holds:
	\begin{equation}
		\sum_{n=1}^RI_n\leq T_{\a,M_{\S}}^{Z^{u\S}}\leq\sum_{n=1}^{R+1}I_n.
	\end{equation}
	%---------------------------------
The underlying process, the sum process $\S(t)$ \eqref{sumprocess}, of $Z_n^{u\S}$ jumps with a rate 
	\begin{align}
		\lambda_{\S}=f\S(t)K+D\S(t)K+c\S(t)^2K+\D N_{aa}(t)\leq C_\lambda K.
	\end{align}
Since the Markov chain $Z_n^{u\S}$ has to jump at least one time, it holds that, for all $j\in\N, I_j>W$, a.s., where $W\sim\exp(C_\lambda K)$. Thus
	\begin{equation}
		\P[I_j<y]\leq\P[W<y]=1-\exp(-C_\lambda  Ky).
	\end{equation}
We have
	\begin{align}
		\P[T_{\a,M_{\S}}^{Z^{u\S}}<\t_{aA}^{i+}\land\t_{aA}^{(i+1)-}\land e^{VK^{\a}}]
		&\leq\P[T_{\a,M_{\S}}^{Z^{u\S}}<e^{VK^{\a}}]\nonumber\\
		&\leq\P[T_{\a,M_{\S}}^{Z^{u\S}}< e^{VK^{\a}}\cap\{R>N\}]%\nonumber\\&\quad
		+\P[T_{\a,M_{\S}}^{Z^{u\S}}< e^{VK^{\a}}\cap\{R\leq N\}]\nonumber\\
		&\leq\P[T_{\a,M_{\S}}^{Z^{u\S}}< e^{VK^{\a}}\cap\{R>N\}]+\P[R\leq N].
	\end{align}
First we estimate $\P[T_{\a,M_{\S}}^{Z^{u\S}}< e^{VK^{\a}}\cap\{R>N\}]$. Since $T_{\a,M_{\S}}^{Z^{u\S}}\geq\sum_{n=1}^RI_n$, it holds that if $\frac{n}{2}$ of the $I_j$ are greater than $\frac{2}{n}e^{VK^{\a}}$, then $T_{\a,M_{\S}}^{Z^{u\S}}\geq e^{VK^{\a}}$.
	\begin{align}
		\P[T_{\a,M_{\S}}^{X,u\S}<e^{VK^{\a}}\cap\{R>N\}]
		&\leq\sum_{n=N}^\infty\P[T_{\a,M_\S}^{X,u\S}<e^{VK^{\a}}\cap\{R=n\}]\nonumber\\
		&\leq\sum_{n=N}^\infty\P\left[\sum_{j=1}^nI_j<e^{VK^{\a}}\right]
		\leq\sum_{n=N}^\infty\P\left[\sum_{j=1}^n\1_{\left\{I_j<\frac{2}{n}e^{VK^{\a}}\right\}}>\frac{n}{2}\right].
	\end{align}
We have $\P[I_j<\frac{2}{n}e^{VK^{\a}}]\leq\P[W<\frac{2}{n}e^{VK^{\a}}]=1-\exp\left(-\frac{2C_\lambda K}{n}e^{VK^{\a}}\right)\equiv p_n$. Hence the probability that at least $\frac{n}{2}$ of the $n$ random variables are greater than $\frac{2}{n}e^{VK^{\a}}$ is binomial distributed  with parameters $p_n,n$.
	\be
	\label{1Teil}
		\sum_{n=N}^\infty\P\left[\sum_{j=1}^n\1_{\left\{I_j<\frac{2}{n}e^{VK^{\a}}\right\}}>\frac{n}{2}\right]
		=\sum_{n=N}^\infty\sum_{m=n/2}^n\binom{n} {m}(1-p_n)^{n-m}p_n^m\leq\sum_{n=N}^\infty 4^np_n^{\frac{n}{2}}
		\leq \frac{(16 p_N)^{N/2}}{1-4p_N^{1/2}}.
		\ee
where we used that, in  the  range of summation, $p_n\leq p_N$. 
Then, for $K$ large enough, $4p_N^{1/2}\leq 1/2$, and 
\be
(16 p_N)^{N/2} =\left(16\left(1- \exp\left(-\sfrac{2C_\lambda}{N}Ke^{VK^{\a}}\right)\right)\right)^{N/2}
\leq \left(16\left(\sfrac{2C_\lambda}{N}Ke^{VK^{\a}}\right)\right)^{N/2}
=  \left(16\left(2C_\lambda AK^3e^{VK^{\a}}\right)\right)^{N/2}.
\ee
Recalling that $A=e^{-O(K^{3\a})}$, one sees that \eqref{1Teil} is bounded by $o\left(e^{- K^{2\a}}\right)$.
Hence we get
	\be
		\P\left[T_{\a,M_{\S}}^{X,u\S}<\t_{aA}^{i+}\land\t_{aA}^{(i+1)-}\land e^{VK^{\a}}\right]
		\leq\P\left[T_{\a,M_{\S}}^{X,u\S}<e^{VK^\a}\cap \{R>N\}\right]+\P[R\leq N]
		=o\left(\frac{1}{K}\right).
	\ee
%\end{proof}
This concludes the proof of Proposition \eqref{upSum}.
\end{proof}
%-------------------------------------------
\paragraph{Step 2:}
We derive a rough upper bound on the process $n_{aa}$ (Proposition \ref{ubaa}). Recall that
\begin{align}
\g_\D\equiv\frac{f+\frac\D2}{4\bar n_A(f+\D)}.
\end{align}
\begin{proof}[Proof of Proposition \ref{ubaa}]
The proof is similar to the one  in Step 1. Again, we define the difference process $X_t^{aa}$ between $N_{aa}$ and $\g_\D x^{2i}\e^2K$. This is a branching process with the same rates as $n_{aa}$. Set
	\bea
		X_t^{aa}&=&N_{aa}(t)-\g_\D x^{2i}\e^2K\\
		T_0^{X,aa}&\equiv&\inf\{t\geq0: X_t^{aa}=0\}\\
		T_{\a,M_{aa}}^{X,aa}&\equiv&\inf\{t\geq0: X_t^{aa}\geq 3M_{aa}(x^{2i}\e^2)^{1+\a}K\}.
	\eea
Let $Y_n^{aa}$ be the discretisation of $X_t^{aa}$, obtained as described in Step 1. 
\begin{lemma}
	For $t\in\N$ such that $\vartheta_n\in \left[\t_{aA}^{i-}, \t_{aA}^{i+}\land\t_{aA}^{(i+1)-}\land e^{VK^{\a}}\right]$ and $1\leq k\leq 3M_{aa}(x^{2i}\e)^{1+\a}K$, there exists a constant $C_0>0$ such that
	\begin{equation}
		\P[Y_{n+1}^{aa}=k+1|Y_n^{aa}=k]\leq \frac{1}{2}-C_0\equiv p_{aa}.
	\end{equation}
\end{lemma}
\begin{proof}
In the following we use Proposition \ref{propD} for the first iteration step, since the mutant population $n_{AA}$ increased to an $\e$-neighbourhood of its 
equilibrium $\bar n_A$ and the other two populations decreased to an $\e$ order. Thus the influence of the small $aA$- and $aa$-populations is negligible for 
the dynamics of $n_{AA}$ and the $AA$-population will stay in the $\bar n_A$-neighbourhood an exponentially long time. Now for $i^{th}$ iteration-step we 
use the finer bounds of $n_{AA}(t)$ (Proposition \ref{lbAA} and \ref{ubAA}) and $\S(t)$ (Proposition \ref{lbSum}) estimated in the $(i-1)^{th}$ iteration-step before.
By \eqref{raten2}, we have
	\begin{align}
	\label{nomaa}
		\P[Y_{n+1}^{aa}=k+1|Y_n^{aa}=k]
	&=\frac{(\g_\D x^{2i}\e^2K+k)f\left(1-\frac{n_{AA}(t)}{\S(t)}\right)+\frac{fK}{4\S(t)}n_{aA}^2(t)}{(\g_\D x^{2i}\e^2K+k)\left[f\left(1-\frac{n_{AA}(t)}{\S(t)}\right)+D+\D+c\S(t)\right]+\frac{fK}{4\S(t)}n_{aA}^2(t)}\nonumber\\
		&=\frac{1}{2}+\frac{\frac{1}{2}f-\frac{f}{2\S(t)}n_{AA}(t)+\frac{fK}{8\S(t)}\frac{n_{aA}^2(t)}{(\g_\D x^{2i}\e^2K+k)}-\frac{1}{2}D-\frac{1}{2}\D-\frac{1}{2}c\S(t)}{f\left(1-\frac{n_{AA}(t)}{\S(t)}\right)+D+\D+c\S(t)+\frac{fK}{4\S(t)}\frac{n_{aA}^2(t)}{(\g_\D x^{2i}\e^2K+k)}}.
	\end{align}
Using Propositions \ref{ubAA}, \ref{lbSum} , and \ref{upSum} and \eqref{gammadelta}, one sees that the  numerator in the second summand  of \eqref{nomaa} 
is bounded from above by
	\begin{align}
		& \frac{1}{2}f-\frac{f(\bar n_A-x^{i}\e)}{2(\bar n_A+3M_\S (x^{2i}\e^2)^{1+\a})}+\frac{f(x^i\e+x^{2i}\e^2)^2}{8\g_\D x^{2i}\e^2(\bar n_A-\frac{\D+\vartheta}{c\bar n_A}\g_\D x^{2i}\e^2+3M_{\S}(x^{2i}\e^2)^{1+\a})}
		-\frac{D+\D+c\bar n_A}{2}+\OO(x^{2i}\e^2) \nonumber\\
		&\leq - \frac{f+\D}{2}+\frac{f}{8\bar n_A\g_\D}+\mathcal{O}(x^{i}\e)
		=-\frac{f+\D}{2f+\D}\frac\D2+\mathcal{O}(x^{i}\e).
	\end{align}
Since $\e<\frac\D2$, there exists a constant $C_0>0$ such that
	\begin{align}
		\P[Y_{n+1}^{aa}=k+1|Y_n^{aa}=k]\leq\frac{1}{2}-C_0\equiv p_{aa}.
	\end{align}
\end{proof}
As in Step 1 we couple $Y_n^{aa}$ on a process $Z_n^{aa}$ via:
	\begin{align}
	&(1)\quad Z_0^{aa}=Y_0^{aa}\lor 0\\
	&(2)\quad \P[Z_{n+1}^{aa}=k+1|Y_n^{aa}<Z_n^{aa}=k]=p_{aa},\\
	&(3)\quad \P[Z_{n+1}^{aa}=k-1|Y_n^{aa}<Z_n^{aa}=k]=1-p_{aa},\\
	&(4)\quad \P[Z_{n+1}^{aa}=k+1|Y_{n+1}^{aa}=k+1, Y_n^{aa}=Z_n^{aa}=k]=1,\\
	&(5)\quad \P[Z_{n+1}^{aa}=k+1|Y_{n+1}^{aa}=k-1, Y_n^{aa}=Z_n^{aa}=k]=\frac{p_{0}-\P[Y_{n+1}^{aa}=k+1|Y_n^{aa}=k]}{1-\P[Y_{n+1}^{aa}=k+1|Y_n^{aa}=k]}.
	\end{align}
Observe that by construction $Z_n^{aa}\succcurlyeq Y_n^{aa}$, a.s.. The marginal distribution of $Z_n^{aa}$ is the desired Markov chain with transition probabilities 
	\be
		\P[Z_{n+1}^{aa}=k+1|Z_n^{aa}=k]=p_{aa}=1-
		\P[Z_{n+1}^{aa}=k-1|Z_n^{aa}=k]
	\ee
and invariant measure
	\begin{align}
		\mu(n)=\frac{\prod_{k=1}^{n-1}\left(\frac{1}{2}-C_0\right)}{\prod_{k=1}^{n}\left(\frac{1}{2}+C_0\right)}=\frac{\left(\frac{1}{2}-C_0\right)^{n-1}}{\left(\frac{1}{2}+C_0\right)^{n}}.
	\end{align}
The remainder of the proof is a complete re-run of the proof of Proposition \ref{ubaa} and we skip the details.
\end{proof}
%-----------------------------------------
\paragraph{Step 3:}
We estimate the lower bound on the process $\S(t)$, for $t\in \left[\t_{aA}^{i-}, \t_{aA}^{i+}\land\t_{aA}^{(i+1)-}\land e^{VK^{\a}}\right]$.
\begin{proof}[Proof of Proposition \ref{lbSum}]
The proof is similar to those in Step 1 and 2. We only perform the crucial steps.
This time the difference process is given by the difference of $\S(t)$ and $\bar n_A-\frac{\D+\vartheta}{c\bar n_A}\g_\D x^{2i}\e^2$. Let 
	\bea
		X_t^{l\S}&=&\S(t)K-\left(\bar n_A-\sfrac{\D+\vartheta}{c\bar n_A}\g_\D x^{2i}\e^2 \right)K\\
		T_0^{X,l\S}&\equiv&\inf\{t\geq0: X_t^{l\S}=0\}\\
		T_{\a,M_\S}^{X,l\S}&\equiv&\inf\{t\geq0: X_t^{l\S}\leq- 3M_{\S}(x^{2i}\e^2)^{1+\a}K\}.
	\eea
As described in Step 1 we construct the discrete process $Y_n^{l\S}$ associated to $X_t^{l\S}$. We show that $Y_n^{l\S}$ jumps down with a probability less than $\frac12$:
\begin{lemma}
	For $t\in\N$ such that $\vartheta_n\in \left[\t_{aA}^{i-}, \t_{aA}^{i+}\land\t_{aA}^{(i+1)-}\land e^{VK^{\a}}\right]$ and $1\leq k\leq 3M_\S (x^{2i}\e^2)^{1+\a}K$ there exists  a constant $C_0>0$ such that
	\begin{equation}
		\P[Y_{n+1}=-k-1|Y_n=-k]\leq \frac{1}{2}-C_0x^{2i}\e^2\equiv p_{\S}.
	\end{equation}
\end{lemma}

\begin{proof}
Using the rates of the sum process \eqref{raten1} and the upper bound on $n_{aa}$ obtained in Step 2, this is a simple computation and we skip the details.
\end{proof}
As in Step 1, to obtain a Markov chain we couple the process $Y_n^{l\S}$ on a 
process $Z_n^{l\S}$ via:
	\begin{align}
	&(1)\quad Z_0=Y_0\lor 0,\\
	&(2)\quad \P[Z_{n+1}^{l\S}=-k+1|Y_n^{l\S}>Z_n^{l\S}=-k]=1-p_{\S},\\
	&(3)\quad \P[Z_{n+1}^{l\S}=-k-1|Y_n^{l\S}>Z_n^{l\S}=-k]=p_{\S},\\
	&(4)\quad \P[Z_{n+1}^{l\S}=-k-1|Y_{n+1}^{l\S}=-k-1, Y_n^{l\S}=Z_n^{l\S}=-k]=1,\\
	&(5)\quad \P[Z_{n+1}^{l\S}=-k-1|Y_{n+1}^{l\S}=-k+1, Y_n^{l\S}=Z_n^{l\S}=-k]=\sfrac{p_{\S}-\P[Y_{n+1}^{l\S}=-k-1|Y_n^{l\S}=-k]}{1-\P[Y_{n+1}^{l\S}=-k-1|Y_n^{l\S}=-k]}.
	\end{align}
Observe that by construction $Z_n^{l\S}\preccurlyeq Y_n^{l\S}$, a.s.. The marginal distribution of $Z_n^{l\S}$ is the desired Markov chain with transition probabilities 
	\begin{align}
		&\P[Z_{n+1}^{l\S}=k+1|Z_n^{l\S}=k]=p_{\S},\\
		&\P[Z_{n+1}^{l\S}=k-1|Z_n^{l\S}=k]=1-p_{\S},
	\end{align}
and invariant measure
	\begin{align}
		\mu(n)=\frac{\prod_{k=1}^{n-1}\left(\frac{1}{2}-C_0x^{2i}\e^2\right)}{\prod_{k=1}^{n}\left(\frac{1}{2}+C_0x^{2i}\e^2\right)}=\frac{\left(\frac{1}{2}-C_0x^{2i}\e^2\right)^{n-1}}{\left(\frac{1}{2}+C_0x^{2i}\e^2\right)^{n}}.
		\end{align}
The remainder of the proof follows along the lines of the proof given in Step 1. We prove that the process returns to zero many times before it hits $3M_{\S}(x^{2i}\e^2)^{1+\a}$ and calculate the duration of  
 one zero-return to get the desired result.
\end{proof}
%----------------------------------------
\paragraph{Step 4:}
With the results of the Steps 1-3 and the settings we are able to calculate a lower bound (Step 4.1) and an upper bound (Step 4.2) for $n_{AA}(t)$, for $t\in \left[\t_{aA}^{i-}, \t_{aA}^{i+}\land\t_{aA}^{(i+1)-}\land e^{VK^{\a}}\right]$.
\\[6pt]
\paragraph{Step 4.1:} We now prove Proposition \ref{lbAA}, the lower bound on $n_{AA}$.
\begin{proof}[Proof of Proposition \ref{lbAA}]
From Step 3 we know that $\S(t)\geq\bar n_A-\frac{\D+\vartheta}{c\bar n_A}\g_\D x^{2i}\e^2-3M_{\S}(x^{2i}\e^2)^{1+\a}$. With the upper bound in Step 2 for $n_{aa}$ and the settings for the steps used for the $aA$-population, we get
	\begin{align}
		n_{AA}(t)&= \S(t)-n_{aA}(t)-n_{aa}(t)\nonumber\\
		&\geq \bar n_A-\frac{\D+\vartheta}{c\bar n_A}\g_\D x^{2i}\e^2-x^i\e-(1+\g_\D) x^{2i}\e^2-3(M_{\S}+M_{aa})(x^{2i}\e^2)^{1+\a}\nonumber\\
		&\geq\bar n_A-x^i\e+\mathcal{O}(x^{2i}\e^2).
	\end{align}
\end{proof}
\paragraph{Step 4.2:} We prove Proposition \ref{ubAA}, the upper bound on $n_{AA}$.
\begin{proof}[Proof of Proposition \ref{ubAA}]
From Step 1 we know that $\S(t)\leq\bar n_A+3M_{\S}(x^{2i}\e^2)^{1+\a}$. With the lower bound on $n_{aA}(t)$ defined in the settings we get
	\be
		n_{AA}(t)= \S(t)-n_{aA}(t)-n_{aa}(t)\leq \bar n_A+3M_{\S}(x^{2i}\e^2)^{1+\a}-x^{i+1}\e\leq\bar n_A-x^{i+1}\e+\mathcal{O}((x^{2i}\e^2)^{1+\a}).
	\ee
\end{proof}

\paragraph{Step 5:}
Up to now we have estimated upper and lower bounds for all single processes: $\S(t), n_{aa}(t), n_{aA}(t)$ and $n_{AA}(t)$, for $t\in \left[\t_{aA}^{i-}, \t_{aA}^{i+}\land\t_{aA}^{(i+1)-}\land e^{VK^{\a}}\right]$.
Now, we prove that $n_{aA}(t)$ has the tendency to decrease on the time intervals defined in the settings. For this we restart $n_{aA}$ when the process hits $x^i\e$ and show that with high probability the $aA$-
population decreases to $x^{i+1}\e$ before it exceeds again the $x^i\e+x^{2i}\e^2$-value (Proposition \ref{decayaA}).  For this we couple $n_{aA}(t)$ on a process which minorise it and on one which majorise it 
and show that these processes decrease to $x^{i+1}\e$ before going back to $x^i\e+x^{2i}\e^2$.
\begin{proof}[Proof of Proposition \ref{decayaA}]
As before let $Y_n^{aA}$ (cf. e.g. Step 1) be the associated discrete process to $N_{aA}(t)$. We start by coupling $Y_n^{aA}$ on a Markov chain $Z_n$ such that $Z_n^u\succcurlyeq Y_n^{aA}$, a.s..
\begin{lemma}
\label{decaylemaA}
		For $t\in\N$ such that $\vartheta_n\in \left[\t_{aA}^{i-}, \t_{aA}^{i+}\land\t_{aA}^{(i+1)-}\land e^{VK^{\a}}\right]$ there exists  a constant $C_0>0$ such that
	\begin{equation}
		\P[Y_{n+1}^{aA}=k+1|Y_n^{aA}=k]\leq \frac{1}{2}-C_0x^{i+1}\e\equiv p_{aA}^u.
	\end{equation}
\end{lemma}
\begin{proof}
For $t<\t_{aA}^{i+}\land\t_{aA}^{(i+1)-}\land e^{VK^{\a}}$, we have
	\begin{align}
	\label{nomaA}
		\P[Y_{n+1}^{aA}=k+1|Y_n^{aA}=k]
		&=\frac{fk\left(1-\frac{kK^{-1}}{2\S(t)}\right)+2fN_{aa}(t)\frac{n_{AA}(t)}{\S(t)}}{fk\left(1-\frac{kK^{-1}}{2\S(t)}\right)+2fN_{aa}(t)\frac{n_{AA}(t)}{\S(t)}+k[D+c\S(t)]}\nonumber\\
		&\leq\frac{1}{2}+\frac{\frac{1}{2}f-\frac{f}{4\S(t)}kK^{-1}+\frac{f n_{AA}(t)}{k\S(t)}N_{aa}(t)-\frac{1}{2}D-\frac{1}{2}c\S(t)}{f\left(1-\frac{kK^{-1}}{2\S(t)}\right)+2\frac{fN_{aa}(t)}{k}\frac{n_{AA}(t)}{\S(t)}+[D+c\S(t)]}.
	\end{align}
As in the previous steps, we bound the  nominator of the second summand in  \eqref{nomaA} using  Propositions \ref{upSum}, \ref{lbSum}, and \ref{ubAA}, from above by 
	\begin{align}
		&\sfrac{1}{2}f-\sfrac{f}{4\bar n_A}kK^{-1}+\sfrac{f(\bar n_A-x^{i+1}\e)}{kK^{-1}(\bar n_A-\sfrac{\D+\vartheta}{c\bar n_A}\g_\D x^{2i}\e^2)}\g_\D x^{2i}\e^2-\sfrac{1}{2}D-\sfrac{1}{2}c(\bar n_A-\sfrac{\D+\vartheta}{c\bar n_A}\g_\D x^{2i}\e^2)+\mathcal{O}((x^{2i}\e^2)^{1+\a})\nonumber\\
		&\leq-\frac f{4\bar n_A}\frac{\vartheta-\frac\D2}{f+\vartheta}x^{i+1}\e+\OO(x^{2i}\e^2).
	\end{align}
This term is negative since $\frac\D2<\vartheta$. Hence, we get
	\begin{align}
		\P[Y_{n+1}^{aA}=k+1|Y_n^{aA}=k]\leq\frac{1}{2}-\frac{\vartheta-\frac\D2}{8\bar n_A(f+\vartheta)}x^{i+1}\e+\OO(x^{2i}\e^2)\equiv p_{aA}^u.
	\end{align}
\end{proof}

To obtain a Markov chain we couple the process $Y_n^{aA}$ on a process $Z_n^u$ via:
	\begin{align}
	&(1)\quad Z_0^u=Y_0^{aA}\\
	&(2)\quad \P[Z_{n+1}^u=k+1|Y_n^{aA}<Z_n^u=k]=p_{aA}^u,\\
	&(3)\quad \P[Z_{n+1}^u=k-1|Y_n^{aA}<Z_n^u=k]=1-p_{aA}^u,\\
	&(4)\quad \P[Z_{n+1}^u=k+1|Y_{n+1}^{aA}=k+1, Y_n^{aA}=Z_n^u=k]=1,\\
	&(5)\quad \P[Z_{n+1}^u=k+1|Y_{n+1}^{aA}=k-1, Y_n^{aA}=Z_n^u=k]=\frac{p_{aA}^u-\P[Y_{n+1}^{aA}=k+1|Y_n^{aA}=k]}{1-\P[Y_{n+1}^{aA}=k+1|Y_n^{aA}=k]}.
	\end{align}
Observe that by construction $Z_n^u\succcurlyeq Y_n^{aA}$, a.s.. The marginal distribution of $Z_n^u$ is the desired Markov chain with transition probabilities 
	\begin{align}
		&\P[Z_{n+1}^u=k+1|Z_n^u=k]=p_{aA}^u,\\
		&\P[Z_{n+1}^u=k-1|Z_n^u=k]=1-p_{aA}^u,
	\end{align}
and invariant measure
	\begin{align}
		\mu(n)=\frac{\prod_{k=1}^{n-1}\left(\frac{1}{2}-C_0x^{i+1}\e\right)}{\prod_{k=1}^{n}\left(\frac{1}{2}+C_0x^{i+1}\e\right)}=\frac{\left(\frac{1}{2}-C_0x^{i+1}\e\right)^{n-1}}{\left(\frac{1}{2}+C_0x^{i+1}\e\right)^{n}}.
	\end{align}
We define the stopping times
\begin{align}
	&T_{i+}^Z\equiv\inf\{\vartheta_n\geq0: Z_n\geq x^i\e K+x^{2i}\e^2 K\},\\
	&T_{(i+1)-}^Z\equiv\inf\{\vartheta_n\geq0: Z_n\leq x^{i+1}\e K\}.
\end{align}
For $x^{i+1}\e\leq z< x^i\e$, we get as before the following bound on the harmonic function
\be
		\P_{zK}[T_{i+}^{Z^u}<T_{(i+1)-}^{Z^u}]
		=\frac{\sum_{n=x^{i+1}\e K+1}^{zK}\left(\frac{1+2C_0x^{i+1}\e}{1-2C_0x^{i+1}\e}\right)^{n-1}}{\sum_{n=x^{i+1}\e K+1}^{x^{i}\e K+x^{2i}\e^2 K}\left(\frac{1+2C_0x^{i+1}\e}{1-2C_0x^{i+1}\e}\right)^{n-1}}
\nonumber\\
		\leq K^{1/4-\a}\exp\left(-\tilde CK^{1/4+3\a}\right).
	\ee
Now we couple $Y_n^{aA}$ on a Markov chain $Z_n^l$ such that $Z_n^l\preccurlyeq Y_n^{aA}$, a.s..
\begin{lemma}
\label{decaylemaAl}
		For $t\in\N$ such that $\vartheta_n\in \left[\t_{aA}^{i-}, \t_{aA}^{i+}\land\t_{aA}^{(i+1)-}\land e^{VK^{\a}}\right]$ there exists  a constant $C_1>0$ such that
	\begin{equation}
		\P[Y_{n+1}=k+1|Y_n^{aA}=k]\geq \frac{1}{2}-C_1x^{i}\e\equiv p_{aA}^l.
	\end{equation}
\end{lemma}
\begin{proof} The proof is completely analogous to the proof of Lemma \ref{decaylemaA}and we skip the details.
\end{proof}
To obtain a Markov chain we couple the process $Y_n^{aA}$ on a process $Z_n^l$ via:
	\begin{align}
	&(1)\quad Z_0^l=Y_0\\
	&(2)\quad \P[Z_{n+1}^l=k+1|Y_n^{aA}>Z_n^l=k]=p_{aA}^l,\\
	&(3)\quad \P[Z_{n+1}^l=k-1|Y_n^{aA}>Z_n^l=k]=1-p_{aA}^l,\\
	&(4)\quad \P[Z_{n+1}^l=k-1|Y_{n+1}=k-1, Y_n^{aA}=Z_n^l=k]=1,\\
	&(5)\quad \P[Z_{n+1}^l=k-1|Y_{n+1}=k+1, Y_n^{aA}=Z_n^l=k]=\frac{\P[Y_{n+1}=k+1|Y_n^{aA}=k]-p_{aA}^l}{\P[Y_{n+1}=k+1|Y_n^{aA}=k]}.
	\end{align}
Observe that by construction $Z_n^l\preccurlyeq Y_n^{aA}$, a.s.. The marginal distribution of $Z_n^l$ is the desired Markov chain with transition probabilities 
	\begin{align}
		&\P[Z_{n+1}^l=k+1|Z_n^l=k]=p_{aA}^l,\\
		&\P[Z_{n+1}^l=k-1|Z_n^l=k]=1-p_{aA}^l.
	\end{align}
Similar to the upper process, we can show that  the lower process reaches $x^{i+1}\e K$ before returning to $x^i\e K+x^{2i}\e^2 K$, with high probability. This concludes the proof of the proposition.
\end{proof}
%---------------------------------------
\paragraph{Step 6:}
In this step we calculate the time which $n_{aA}(t)$ needs to decrease from $x^{i}\e$ to $x^{i+1}\e$ (Proposition \ref{timeaA}).
\begin{proof}[Proof of Proposition \ref{timeaA}]
	Let $Z_n^l\preccurlyeq Y_n^{aA}\preccurlyeq Z_n^u$ be defined as in the step before and $Y_0=Z_0^l=Z_0^u=x^{i}\e K$.\\
	Recalling \eqref{raten3}, we get
\begin{align}
		\lambda_{aA}
		&=fN_{aA}(t)\left(1-\frac{n_{aA}(t)}{2\S(t)}\right)+2fN_{aa}(t)\frac{n_{AA}(t)}{\S(t)}+N_{aA}(t)[D+c\S(t)]\nonumber\\
		&\geq 2fx^{i+1}\e K+\OO(x^{2i}\e^2K)\equiv C_\lambda x^{i+1}\e K\equiv\lambda_{aA}^l,
	\end{align}
	\begin{align}
		\lambda_{aA}
		&=fN_{aA}(t)\left(1-\frac{n_{aA}(t)}{\S(t)}\right)+2fN_{aa}(t)\frac{n_{AA}(t)}{\S(t)}+N_{aA}(t)[D+c\S(t)]\nonumber\\
		&\leq 2fx^{i}\e K+\OO(x^{2i}\e^2K)\equiv C_\lambda x^{i}\e K\equiv \lambda_{aA}^u.
	\end{align}
	Let $n_*:=\inf\{n\geq0:Y_n^{aA}-Y_0^{aA}\leq-(1-x)x^i\e K\}$ the random variable which counts the number of jumps $Y_n^{aA}$ makes until it is smaller than $-(1-x)x^i\e K$. The time between two jumps of $n_{aA}(t)$ is given by $\tau_m-\tau_{m-1}$. 
	It holds that 
	$J_m^u\preceq\tau_m-\tau_{m-1}\preceq J_m^l$, where $J_m^l$ (resp. $J_m^u$) are i.i.d. exponential distributed random variables with parameter $\lambda_{aA}^l$ (resp. $\lambda_{aA}^u$). 
	We want to estimate bounds for the times that the processes $Z_n^u$, resp. $Z_n^l$, need to decrease from $x^{i}\e K$  to $x^{i+1}\e K$. Thus we show, for constants $C_u,C_l>0$, that 
	\begin{align}
	&\text(i) \quad\P\left[\sum_{m=1}^{n_*}J_m^u>\frac{2C_u}{C_\lambda x^{i+1}\e}\right]\leq\exp(-MK^{1/2+2\a}),\label{thetai}\\
	&\text(ii) \quad\P\left[\sum_{m=1}^{n_*}J_m^l<\frac{C_l}{2C_\lambda x^{i}\e}\right]\leq \exp(-MK^{1/2+2\a}).\label{thetaii}
	\end{align}
	We start by showing \eqref{thetai}
We   need to find  $N$ such that, with high probability, $n_*\leq N$. To do this, we use the 	majorising process $Z^u$.
 Let $W_k^u$ be i.i.d. random variables with
		\begin{align}
		 	\P[W_k^u=1]=\frac{1}{2}-C_0x^{i+1}\e, \quad\P[W_k^u=-1]=\frac{1}{2}+C_0x^{i+1}\e,\quad\text{and}\quad\E[W_1^u]=-2C_0x^{i+1}\e.
		\end{align}
 $W_k^u$ records  a birth or a death event of the process $Z^u_n$.   We get
\begin{align}
\P[n_*\leq N]%&\geq\P\left[\inf\left\{n\geq0: Z_n^u-Z_0^u\leq-\left\lfloor(1-x)x^i\e K\right\rfloor\right\}\leq N\right]\\
&\geq\P\left[\exists n\leq N: \sum_{k=1}^nW_k\leq-\left\lfloor(1-x)x^i\e K\right\rfloor\right] \nonumber\\
&\geq\P\left[\sum_{k=1}^NW_k\leq-\left\lfloor(1-x)x^i\e K\right\rfloor\right]\nonumber\\
&\geq  1-\P\left[\sum_{k=1}^N(W_k-\E W_k)\geq  2 NC_0 x^{i+1}\e   -  \left\lfloor(1-x)x^i\e K\right\rfloor\right].
\end{align}
		By  Hoeffding's inequality and choosing  $N=\frac{1-x}{C_0x}K=:C_uK$, we get
\be
\P[n_*\leq  C_uK]\geq  1-\exp\left(-\frac {\left(x^i\e\right)^2 K x(1-x)}2\right)\geq 1-\exp\left(-{K^{1/2+2\a}C_0(1-x)x}/2\right),
\ee
where we used that $x^{i}\e\geq K^{-1/4+\a}$.	
	Thus
	\begin{align}
\P\left[\sum_{m=1}^{n_*}J_m^u>\frac{2C_u}{C_\lambda x^{i+1}\e}\right]&\leq\P\left[\sum_{m=1}^{C_uK}J_m^u>\frac{2C_u}{C_\lambda x^{i+1}\e}\right]+\exp(-K^{1/2+2\a}C_0x(1-x)/2).
\end{align}
By applying the exponential Chebyshev inequality  we get
\be
\P\left[\sum_{m=1}^{C_uK}J_m^u>\frac{2C_u}{C_\lambda x^{i+1}\e}\right]
\leq\exp\left(- C_uK/2\right).
\ee
Next we  show \eqref{thetaii}. For this we need to find $N$ such that $\P[n_*\leq N]$ is very small. For this we use the process $Z^l$.  Let $W_k^l$ be i.i.d. random variables which record a birth or a death event of the process  $Z_n^l$. They satisfy
		\begin{align}
		 	\P[W_k^l=1]=\frac{1}{2}-C_1x^{i}\e, \quad\P[W_k^l=-1]=\frac{1}{2}+C_1x^{i}\e\quad\text{and}\quad\E[W_1^l]=-2C_1x^{i}\e.
		\end{align}
Note that 
	\begin{align}
\P[n_*\leq N]&\leq\P\left[\inf\left\{n\geq0: Z_n^l-Z_0^l\leq-\left\lfloor(1-x)x^i\e K\right\rfloor\right\}\leq N\right]\\
&=\P\left[\exists n\leq N: \sum_{k=1}^n(W_k-\E W_k)\leq 2nC_1x^i\e  -  \left\lfloor(1-x)x^i\e K\right\rfloor\right]\nonumber
\\
&\leq \sum_{n=0}^N\P\left[\sum_{k=1}^n(W_k-\E W_k)\leq 2nC_1x^i\e  -  \left\lfloor(1-x)x^i\e K\right\rfloor\right] .
\end{align}
If we choose  $N=\frac{1-x}{4C_1}K=:C_lK$, using 
 Hoeffding's inequality, we get, for all $n\leq N$,
\begin{align}
 \P\left[\sum_{k=1}^n (W_k-\E W_k)\geq\left\lceil(1-x)x^i\e K\right\rceil-\frac12(1-x)x^i\e K\right]
\leq\exp(-K^{1/2+2\a}C_1(1-x)/2). 
 \end{align}
 Thus
\begin{align}
\P\left[\sum_{m=1}^{n_*}J_m^l<\frac{C_l}{2C_\lambda x^i\e}\right]\leq\P\left[\sum_{m=1}^{C_lK}J_m^l<\frac{C_l}{2C_\lambda x^i\e}\right]+C_lK\exp\left(-K^{1/2+2\a}C_1(1-x)/2\right).
\end{align}
It holds that
\begin{align}
\P\left[\sum_{m=1}^{C_lK}J_m^l>\frac{C_l}{2C_\lambda x^i\e}\right]=1-\P\left[\sum_{m=1}^{C_lK}J_m^l<\frac{C_l}{2C_\lambda x^i\e}\right]=1-\P\left[-\sum_{m=1}^{C_lK}J_m^l>-\frac{C_l}{2C_\lambda x^i\e}\right].
\end{align}
A simple use of  the exponential Chebyshev inequality shows that 
\begin{align}
\P\left[\sum_{m=1}^{C_lK}J_m^l<\frac{C_l}{2C_\lambda x^i\e}\right]\leq\exp\left(-C_lK/2\right).
\end{align}
Thus we have that  $\P\left[\sum_{m=1}^{n_*}J_m^l<\frac{C_l}{2C_\lambda x^i\e}\right]\leq \exp\left(-MK^{1/2+2\a}\right)$, for some constant  $M>0$.
\end{proof}
%-----------------------------------------
\paragraph{Step 7:}
In this step it is shown that $n_{aa}(t)$ decreases under the upper bound $\g_\D x^{2i+2}\e^2+M_{aa}(x^{2i+2}\e^2)^{1+\a}$, which we need to proceed the next iteration step, in at least the time $n_{aA}(t)$ needs to decrease from $x^i\e$ to $ x^{i+1}\e$ (Proposition \ref{timeaa}). We set the time to zero when $n_{aA}(t)$ hits $x^i\e$. Hence $n_{aa}(0)\leq\g_\D x^{2i}\e^2+M_{aa}(x^{2i}\e^2)^{1+\a}$.  Let
\begin{align}
	&\theta^+_i(aa)\equiv\inf\left\{t\geq0:n_{aa}(t)\geq\g_\D x^{2i}\e^2+3M_{aa}(x^{2i}\e^2)^{1+\a}\right\},\\
	&\theta^-_i(aa)\equiv\inf\left\{t\geq0:n_{aa}(t)\leq\g_\D x^{2i+2}\e^2\right\}.
\end{align}
For the proof we proceed in three parts. First we show that $n_{aa}(t)$ has the tendency to decrease. For this we construct a majorising process for $n_{aa}(t)$ and show that this process decreases on the given time interval. This process is used in the second part to estimate an upper bound on the time which the $aa$-population needs for the decay from $\g_\D x^{2i}\e^2+M_{aa}(x^{2i}\e^2)^{1+\a}$ to $\g_\D x^{2i+2}\e^2$. As result we will get that $n_{aa}(t)$ reaches the next upper bound before the $aA$-population decreases to $x^{i+1}\e$. Thus in the third part we ensure that $n_{aa}(t)$ stays below the bound until $n_{aA}(t)$ reaches $x^{i+1}\e$.\\[0.3cm]
%----------------
\paragraph{Part 1:} We show
\begin{proposition}
For $t\in\left[\tau_{aA}^{i-},\tau_{aA}^{i+}\land\tau_{aA}^{(i+1)-}\land e^{VK^\a}\right]$ there are constants $\bar C,\tilde C>0$ such that
		\begin{align}
			\P\left[\theta^+_i(aa)<\theta^-_i(aa)|n_{aa}(0)\leq \g_\D x^{2i}\e^2+M_{aa}(x^{2i}\e^2)^{1+\a}\right]
			%\leq\frac{\bar C}{(x^i\e)^{2\a}}\exp\left(-\tilde C(x^{2i}\e^2)^{1+\a} K\right)\nonumber\\
			\leq\bar CK^{\a/2}\exp(-\tilde CK^{1/2+\a}).
		\end{align}
\end{proposition}
In this part the same strategy as in Step 5 is used. We couple $n_{aa}(t)$ on a process which  majorise it and show that this process  decrease from 
$\g_\D x^{2i}\e^2+M_{aa}(x^{2i}\e^2)^{1+\a}$ to  $\g_\D x^{2i+2}\e^2$ with high probability. 
\begin{proof}
As before, let $Y_n^{aa}$ be the discretisation of $N_{aa}(t)$. 
We start with the construction of the upper process.
\begin{lemma}
		For $t\in\N$ such that $\vartheta_n\in\left[\tau_{aA}^{i-},\tau_{aA}^{i+}\land\tau_{aA}^{(i+1)-}\land e^{VK^\a}\right]$ there exists  a constant $C_u>0$ such that
	\begin{equation}
		\P[Y^{aa}_{n+1}=k+1|Y_n^{aa}=k]\leq \frac{1}{2}-C_u\equiv p_u.
	\end{equation}
\end{lemma}
\begin{proof}
We know that, for $t\in\left[\tau_{aA}^{i-},\tau_{aA}^{i+}\land\tau_{aA}^{(i+1)-}\land e^{VK^\a}\right]$, 
$n_{aA}(t)\in\left[x^{i+1}\e,x^i\e+x^{2i}\e^2\right]$. Again with \eqref{raten2} we have
	\begin{align}
		\P[Y^{aa}_{n+1}=k+1|Y_n=k]
		&=\frac{fk\left(1-\frac{n_{AA}(t)}{\S(t)}\right)+\frac{fK}{4\S(t)}n_{aA}^2(t)}{fk\left(1-\frac{n_{AA}(t)}{\S(t)}\right)+\frac{fK}{4\S(t)}n_{aA}^2(t)+k[D+\D+c\S(t)]}\nonumber\\
		&\leq\frac{1}{2}+\frac{\frac{1}{2}f-\frac{f}{2\S(t)}n_{AA}(t)+\frac{fK}{8\S(t)}\frac{n_{aA}^2(t)}{k}-\frac{1}{2}D-\frac{1}{2}c\S(t)-\frac{1}{2}\D}{f\left(1-\frac{n_{AA}(t)}{\S(t)}\right)+\frac{fK}{4\S(t)}\frac{n_{aA}^2(t)}{k}+[D+\D+c\S(t)]}.
	\end{align}
Using Proposition \ref{upSum}, \ref{lbSum} and \ref{ubAA}, we bound  the numerator in the second summand 
from above by
	\be
		-\frac{f}{2\bar n_A}(\bar n_A-x^i\e)+\frac{f\left(1+2x^i\e\right)}{8\bar n_A\g_\D x^2}-\frac{\D}{2}+\mathcal{O}(x^{2i}\e^2)
		\leq-\frac{f(f+\D)(\vartheta-\frac\D2)+\D\vartheta(f+\D)}{2(f+\frac\D2)(f+\vartheta)}+\mathcal{O}(x^{i}\e).
	\ee
Since $\frac\D2<\vartheta$ there exists a constant $C_u>0$ such that
	\begin{align}
		\P[Y_{n+1}=k+1|Y_n=k]\leq\frac{1}{2}-C_u\equiv p_{u}.
	\end{align}
\end{proof}
By replacing $p_{aa}$ by $p_{u}$, we couple $Y_n^{aa}$ on the same way to a process $Z_n^u$  as it was done in Step 2. Observe that by construction $Z_n^u\succcurlyeq Y_n^{aa}$, a.s.. The marginal distribution of $Z_n^u$ is the desired Markov chain with transition probabilities 
		\begin{align}
		&\P[Z_{n+1}^u=k+1|Z_n^u=k]=p_{u},\\
		&\P[Z_{n+1}^u=k-1|Z_n^u=k]=1-p_{u},
	\end{align}
and invariant measure
	\begin{align}
		\mu(n)=\frac{\prod_{k=1}^{n-1}\left(\frac{1}{2}-C_u\right)}{\prod_{k=1}^{n}\left(\frac{1}{2}+C_u\right)}=\frac{\left(\frac{1}{2}-C_u\right)^{n-1}}{\left(\frac{1}{2}+C_u\right)^{n}}.
	\end{align}
We define the stopping times
\begin{align}
		&T^{+}_i(aa)\equiv\inf\left\{t\geq0:Z_n^u\geq \g_\D x^{2i}\e^2K+3M_{aa}(x^{2i}\e^2)^{1+\a}K \right\},\\
		&T^{-}_i(aa)\equiv\inf\left\{t\geq0:Z_n^u\leq\g_\D x^{2i+2}\e^2K\right\}.
	\end{align}
Again with the formula of the equilibrium potential we estimate for
 $\g_\D x^{2i+2}\e^2K\leq zK<\g_\D x^{2i}\e^2K+M_{aa}(x^{2i}\e^2)^{1+\a}K$
\be
		\P_{zK}[T^{+}_i(aa)<T^{-}_i(aa)]
		=\frac{\sum_{n=\g_\D x^{2i+2}\e^2K+1}^{zK}\left(\frac{1+2C_u}{1-2C_u}\right)^{n-1}}{\sum_{n=\g_\D x^{2i+2}\e^2K+1}^{\g_\D x^{2i}\e^2K+3M_{aa}(x^{2i}\e^2)^{1+\a}K }\left(\frac{1+2C_u}{1-2C_u}\right)^{n-1}}
		\leq \bar CK^{\a/2}\exp(-\tilde CK^{1/2+\a}).
	\ee
\end{proof}
%------------------
\paragraph{Part 2:}
Similar to Step 6, we calculate an upper bound on the time which $n_{aa}(t)$ needs at most to decrease from $\g_\D x^{2i}\e^2+M_{aa}(x^{2i}\e^2)^{1+\a}$ to 
$\g_\D x^{2i+2}\e^2$.
\begin{proposition}
	Let 
		\begin{align}
			\theta_i(aa)\equiv\inf\left\{t\geq0:n_{aa}(t)\leq \g_\D x^{2i+2}\e^2|n_{aa}(0)=\g_\D x^{2i}\e^2+M_{aa}(x^{2i}\e^2)^{1+\a}\right\},
		\end{align}
	the decay time of $n_{aa}(t)$, for $t\in\left[\tau_{aA}^{i-},\tau_{aA}^{i+}\land\tau_{aA}^{(i+1)-}\land e^{VK^\a}\right]$. Then for all $0\leq i\leq\frac{-\ln(\e K^{1/4-\a})}{\ln(x)}$ there exist finite, positive constants $C_u^{aa}$ and $M$, such that
	\begin{align}
		\P\left[\theta_i(aa)>C_u^{aa}\right]\leq\exp\left(-MK^{1/2+2\a}\right)
	\end{align}
\end{proposition}
\begin{proof}
The proof works like the one in Step 6. We calculate an upper bound on the decay time of the majorising process $Z_n^u$. 
	Let $Y_n^{aa}$ and $Z_n^u$ be defined as in the step before and let $W^u$ be i.i.d. random variables with
		\begin{align}
		 	\P[W^u=1]=\frac{1}{2}-C_u, \quad\P[W^u=-1]=\frac{1}{2}+C_u\quad\text{and}\quad\E[W^u_1]=-2C_u.
		\end{align}
	The $W^u$'s record a birth or a death event of $Z_n^{u}$. Similar as in Step 6, we choose $N=\frac{\g_\D(1-x^2)}{C_u}x^{2i}\e^2K=:\bar C x^{2i}\e^2K$ and show that $\P[n_*\leq \frac{\g_\D(1-x^2)}{C_u}x^{2i}\e^2K]\geq1-\exp(-K^{1/2+2\a}C_u\g_\D(1-x^2)/2)$.  It holds
	\begin{align}
		\lambda_{aa}&=fN_{aa}(t)\left(1-\sfrac{n_{AA}(t)}{\S(t)}\right)+\sfrac{fK}{4\S(t)}n_{aA}^2(t)+N_{aa}(t)[D+\D+c\S(t)]\nonumber\\
		&\geq \sfrac{f}{4\bar n_A}x^{2i+2}\e^2K+(f+\D)\g_\D x^{2i+2}\e^2K\nonumber\\
		&\geq \sfrac{2f+\frac\D2}{4\bar n_A}x^{2i+2}\e^2K\equiv C_\lambda x^{2i+2}\e^2K\equiv\lambda_{aa}^l.
	\end{align}
Again, let $\t_m-\t_{m-1}$ be the time between two jumps of $n_{aa}(t)$ and let $J_m^{aa}$ be i.i.d. exponential random variables with parameter $\lambda_{aa}^l$. As in Step 6, using 
the exponential Chebyshev inequality, we get that
$\P\left[\sum_{m=1}^{\bar Cx^{2i}\e^2K}J_m^{aa}>\frac{2\bar C}{C_\lambda x^2}\right]\leq\exp\left(-\bar CK^{1/2+2\a}\right)$
and hence
\begin{align}
\P\left[\sum_{m=1}^{n_*}J_m^{aa}>\frac{2\bar C}{C_\lambda x^2}\right]\leq\exp\left(-MK^{1/2+2a}\right).
\end{align}

%--------------------------
\paragraph{Part 3:}
We see that $n_{aa}(t)$ reaches $\g_\D x^{2i+2}\e^2$ before $n_{aA}(t)$ decreases to $x^{i+1}\e$. Similar to Step 2 we can show that once $n_{aa}(t)$ hits $\g_\D x^{2i+2}\e^2$ it will stay close to it an exponentially long time and will not exceed $\g_\D x^{2i+2}\e^2+M_{aa}(x^{2i+2}\e^2)^{1+\a}$ again. Thus we can ensure that $n_{aa}(t)$ is below $\g_\D x^{2i+2}\e^2+M_{aa}(x^{2i+2}\e^2)^{1+\a}$ when $n_{aA}(t)$ reaches $x^{i+1}\e$.
\end{proof}
%---------------------
\paragraph{Step 8:}
For the iteration we have to ensure that the sum process increases on the given time interval from the value $\bar n_A-\frac{\D+\vartheta}{c\bar n_A}\g_\D x^{2i}\e^2$ to $\bar n_A-\frac{\D+\vartheta}{c\bar n_A}\g_\D x^{2i+2}\e^2$ and stays there until the next iteration-step (Proposition \ref{timeSum}). We set the time to zero when the $aA$-population hits $x^i\e$. Hence $\S(0)\geq\bar n_A-\sfrac{\D+\vartheta}{c\bar n_A}\g_\D x^{2i}\e^2-M_{\S}(x^{2i}\e^2)^{1+\a} $. As in Step 7 we divide the proof into three parts.\\[6pt]
%---------------------
\paragraph{Part 1:}
Similarly to Part 1 in Step 7, we show that with high probability $\S(t)$ increases to $\bar n_A-\frac{\D+\vartheta}{c\bar n_A}\g_\D x^{2i+2}\e^2$ before going back to $\bar n_A-\sfrac{\D+\vartheta}{c\bar n_A}\g_\D x^{2i}\e^2-3M_{\S}(x^{2i}\e^2)^{1+\a}$ .
We define stopping times on $\S(t)$:
\begin{align}
	&\t_{\S}^{i-}\equiv\inf\left\{t\geq0:\S(t)\leq\bar n_A-\sfrac{\D+\vartheta}{c\bar n_A}\g_\D x^{2i}\e^2-3M_{\S}(x^{2i}\e^2)^{1+\a}\equiv v(\S_-)-3M_{\S}(x^{2i}\e^2)^{1+\a}\right\},\\
	&\t_{\S}^{i+}\equiv\inf\left\{t\geq0:\S(t)\geq\bar n_A-\sfrac{\D+\vartheta}{c\bar n_A}\g_\D x^{2i+2}\e^2\equiv v(\S_+)\right\},
\end{align}
where $M_{\S}>0$.
\begin{proposition}
	There are constants $\bar C,\tilde C>0$ such that
		\begin{align}
			\P[\t_{\S}^{i-}<\t_{\S}^{i+}|\S(0)\geq\bar n_A-\sfrac{\D+\vartheta}{c\bar n_A}\g_\D x^{2i}\e^2+M_\S(x^{2i}\e^2)^{1+\a}]
			&\leq \bar CK^{\a/2} \exp\left(-\tilde CK^{7/2\a+2\a^2}\right).
		\end{align}
\end{proposition}
\begin{proof}
From the step before we know that $n_{aa}(t)$ decreases under $\g_\D x^{2i+2}\e^2+M_{aa}(x^{2i+2}\e^2)^{1+\a}$ in a time of order 1 and does not exceed this bound once it hits it. Thus, with the knowledge of Step 3, we show that, for  $t\in\left[\t_{aA}^{i-}+\theta_i(aa),\t_{aA}^{i+}\land\t_{aA}^{(i+1)-}\right]$, the sum process has the tendency to increase and exceed the lower bound $\bar n_A-\frac{\D+\vartheta}{c\bar n_A}\g_\D x^{2i+2}\e^2$ before $n_{aA}(t)$ hits $x^{i+1}\e$.
As before, let $Y_n^{\S}$ be the associated discrete process to $\S(t)$.
\begin{lemma}
		For $t\in\N$ such that $\vartheta_n\in\left[\t_{aA}^{i-}+\theta_i(aa),\t_{aA}^{i+}\land\t_{aA}^{(i+1)-}\land e^{VK^\a}\right]$ there exists  a constant $C_0>0$ such that
	\begin{equation}
		\P[Y^{\S}_{n+1}=k-1|Y^{\S}_n=k]\leq \frac{1}{2}-C_0x^{2i+2}\e^2\equiv p_{0}(\S).
	\end{equation}
\end{lemma}
\begin{proof}
To show the lemma we use the results of the steps before. It holds:
\begin{align}
	\P[Y^{\S}_{n+1}=k-1|Y^{\S}_n=k]
	&\leq\frac{1}{2}+\frac{-\frac{1}{2}(f-D)+\frac{1}{2}ckK^{-1}+\frac{\D}{2k}N_{aa}(t)}{f+D+ckK^{-1}+\frac{\D}{k}N_{aa}(t)}.
\end{align}
We estimate the nominator
\begin{align}
	\leq	-\frac{f-D}{2}+\frac{1}{2}ckK^{-1}+\frac{\D}{2}\g_\D x^{2i+2}\e^2k^{-1}K+\OO((x^{2i+2}\e^2)^{1+\a}).
\end{align}
This term assumes its maximum at $k=v(\S_+)$.
If we insert this bound, $\bar n_A-\frac{\D+\vartheta}{c\bar n_A}\g_\D x^{2i+2}\e^2$, we can estimate
\begin{align}
	&\leq-\frac{f-D}{2}+\frac{1}{2}c\bar n_A-\frac{\D+\vartheta}{2\bar n_A}\g_\D x^{2i+2}\e^2+\frac{\D}{2\bar n_A}\g_\D x^{2i+2}\e^2+\OO((x^{2i+2}\e^2)^{1+\a})\nonumber\\
	&\leq-\frac{\vartheta\g_\D}{2\bar n_A}x^{2i+2}\e^2+\OO((x^{2i+2}\e^2)^{1+\a}).
\end{align}
Thus, there exists a constant $C_0>0$ such that
\begin{align}
	\P[Y^{\S}_{n+1}=k-1|Y^{\S}_n=k]\leq\frac{1}{2}-C_0x^{2i+2}\e^2.
\end{align}
\end{proof}
Again we couple $Y^{\S}_n$ on $Z_n^l$ via:
\begin{align}
	&(1)\quad Z_0^l=Y^{\S}_0,\\
	&(2)\quad \P[Z_{n+1}^l=k+1|Y^{\S}_n>Z_n^l=k]=1-p_{0}(\S),\\
	&(3)\quad \P[Z_{n+1}^l=k-1|Y^{\S}_n>Z_n^l=k]=p_{0}(\S),\\
	&(4)\quad \P[Z_{n+1}^l=k-1|Y^{\S}_{n+1}=k-1, Y^{\S}_n=Z_n^l=k]=1,\\
	&(5)\quad \P[Z_{n+1}^l=k-1|Y^{\S}_{n+1}=k+1, Y^{\S}_n=Z_n^l=k]=\frac{p_{0}(\S)-\P[Y^{\S}_{n+1}=k-1|Y^{\S}_n=k]}{1-\P[Y^{\S}_{n+1}=k-1|Y^{\S}_n=k]}.
	\end{align}
Observe that by construction $Z_n^l\preccurlyeq Y^{\S}_n$, a.s.. The marginal distribution of $Z_n^l$ is the desired Markov chain with transition probabilities 
	\begin{align}
		&\P[Z_{n+1}^l=k+1|Z_n^l=k]=1-p_{0}(\S),\\
		&\P[Z_{n+1}^l=k-1|Z_n^l=k]=p_{0}(\S),
	\end{align}
and invariant measure
	\begin{align}
		\mu(n)=\frac{\prod_{k=1}^{n-1}\left(\frac{1}{2}+C_0x^{2i+2}\e^2\right)}{\prod_{k=1}^{n}\left(\frac{1}{2}-C_0x^{2i+2}\e^2\right)}=\frac{\left(\frac{1}{2}+C_0x^{2i+2}\e^2\right)^{n-1}}{\left(\frac{1}{2}-C_0x^{2i+2}\e^2\right)^{n}}.
	\end{align}
Again we get a bound on the harmonic function,  for $\bar n_A-\frac{\D+\vartheta}{c\bar n_A}\g_\D x^{2i}\e^2-M_\S(x^{2i}\e^2)^{1+\a}\leq z<\bar n_A-\frac{\D+\vartheta}{c\bar n_A}\g_\D x^{2i+2}\e^2$, 
\be
		\P_{zK}[T_{\S}^{i-}<T_{\S}^{i+}]
		=\frac{\sum_{n=zK+1}^{Kv(\S_+)}\left(\frac{1-2C_0x^{2i+2}\e^2}{1+2C_0x^{2i+2}\e^2}\right)^{n-1}}{\sum_{n=Kv(\S_-)-M_\S (x^{2i}\e^2)^{1+\a}K+1}^{Kv(\S_+)}\left(\frac{1-2C_0x^{2i+2}\e^2}{1+2C_0x^{2i+2}\e^2}\right)^{n-1}}
		\leq \bar CK^{\a/2}\exp\left(-\tilde CK^{7/2\a+2\a^2}\right).
		\ee

\end{proof}
%---------------------
\paragraph{Part 2:}
As in Step 6, to calculate an upper bound on the time which $\S(t)$ need to increase from $\bar n_A-\frac{\D+\vartheta}{c\bar n_A}\g_\D x^{2i}\e^2-M_\S(x^{2i}\e^2)^{1+\a}$ to
$\bar n_A-\frac{\D+\vartheta}{c\bar n_A}\g_\D x^{2i+2}\e^2$, we estimate the number of jumps of the sum process and the duration of one single jump from above on a given time interval. 
%We show that this time is at least so big as the time $n_{aA}(t)$ needs to decrease from $x^i\e$ to $x^{i+1}\e$.
\begin{proposition}
	Let 
		\begin{align}
			\theta_i(\S)\equiv\inf\Big\{t\geq0: &\S(t)\geq\bar n_A-\sfrac{\D+\vartheta}{c\bar n_A}\g_\D x^{2i+2}\e^2|\S(0)=\bar n_A-\sfrac{\D+\vartheta}{c\bar n_A}\g_\D x^{2i}\e^2-M_\S(x^{2i}\e^2)^{1+\a}\Big\},
		\end{align}
	the growth time of $\S(t)$ on the time interval $t\in\left[\t_{aA}^{i-}+\theta_i(aa),\t_{aA}^{i+}\land\t_{aA}^{(i+1)-}\right]$. Then for all $0\leq i\leq\frac{-\ln(\e K^{1/4-\a})}{\ln(x)}$ there exist finite, positive constants, $C_l^\S$ and $M$, such that
	\begin{align}
		\P\left[\theta_i(\S)>C_l^\S\right]\leq\exp\left(-MK^{4\a}\right).
	\end{align}
\end{proposition}
\begin{proof}
	Let $Y^{\S}_n$ be defined as in the step before and let $W_k$ be i.i.d. random variables with
		\begin{align}
		 	\P[W_k=1]=\frac{1}{2}+C_0x^{2i+2}\e^2, \quad\P[W_k=-1]=\frac{1}{2}-C_0x^{2i+2}\e^2\quad\text{and}\quad\E[W_1]=2C_0x^{2i+2}\e^2,
		\end{align}
	which record a birth or a death event of the lower process $Z_n^l$.
	We choose $N=\frac{\D+\vartheta}{4C_0c\bar n_Ax^2}\g_\D(1-x^2)K=:\tilde CK$ and show that $\P\left[n_*\leq \frac{\D+\vartheta}{2C_0c\bar n_Ax^2}\g_\D(1-x^2)K\right]\geq1-\exp\left(-K^{4\a}C_0\g_\D(\D+\vartheta)x^2(1-x^2)/2\right) $.\\
We estimate from above the time the process $\S(t)$ needs to make  one jump:
	\begin{align}
		\lambda_{\S}=f\S(t)K+\S(t)K[D+\D+c\S(t)]+\D N_{aa}(t)\geq C_\lambda K\equiv\lambda_{\S}^l.
	\end{align}
As before let $\t_m-\t_{m-1}$ the time between two jumps of $\S(t)$ and let $J_m^\S$ are i.i.d. exponential random variables with parameter $\lambda_{\S}^l$. As in Step 6, by applying the exponential Chebyshev inequality, we get that
$\P\left[\sum_{m=1}^{\tilde CK}J_m^\S>\frac{2\tilde C}{C_\lambda}\right]\leq\exp\left(-\bar K\tilde C/2\right)$
and hence
\begin{align}
\P\left[\sum_{m=1}^{n_*}J_m^\S>\frac{2\tilde C}{C_\lambda}\right]\leq\exp\left(-MK^{4\a}\right).
\end{align}
\end{proof}
%-------------------
\paragraph{Part 3:}
For the iteration we have to ensure that once $\S(t)$ hits the upper bound $\bar n_A-\frac{\vartheta+\D}{c\bar n_A}\g x^{2i+2}\e^2$ it will stay closed to it for an exponential time on the given time interval. This ensure us that the sum process stays in a small enough neighbourhood of $\bar n_A-\frac{\vartheta+\D}{c\bar n_A}\g x^{2i+2}\e^2$ when $n_{aA}(t)$ hits $x^{i+1}\e$ and the next iteration step can start.
\begin{proposition}
\label{staysum}
Assume that $\S(0)=\bar n_A-\frac{\vartheta+\D}{c\bar n_A}\g x^{2i+2}\e^2$. For all $M>0$, let 
\begin{align}
\hat\tau_{\S}^{\a}\equiv\inf\{t>0: \S(t)-\bar n_A+\frac{\vartheta+\D}{c\bar n_A}\g x^{2i+2}\e^2\leq -M(x^{2i+2}\e^2)^{1+\a}\}.
\end{align}
Then there exists a constant $M_{\S}>0$ such that
\begin{align}
\P[\hat\tau_{\S}^{\a}<\tau_{aA}^{i+}\land\tau_{aA}^{(i+1)-}\land e^{VK^{\a}}]=o(K^{-1}).
\end{align}
\end{proposition}
\begin{proof}
The proof is similar to Step 3. Again we define the difference process $\hat X_t$ between $\S(t)K$ and $\bar n_AK-\frac{\vartheta+\D}{c\bar n_A}\g_\D x^{2i+2}\e^2K$, which is a branching process with the same rates as $\S(t)$:
\bea
\hat X_t&=&\S(t)K-\bar n_AK+\frac{\vartheta+\D}{c\bar n_A}\g_\D x^{2i+2}\e^2K\\
\hat T_0^X&\equiv&\inf\{t\geq0:\hat X_t=0\}\\
\hat T_{\a,M_{\S}}^X&\equiv&\{t\geq0: \hat X_t\leq -M_{\S}(x^{2i}\e^2)^{1+\a}\}.
\eea
Let $\hat Y_n$ be the discrete process associated to $\hat X_t$, obtained as described in Step 1.
\begin{lemma}
For $t\in\N$ such that $\vartheta_n\in\left[\t_{aA}^{i-}+\theta_i(aa),\t_{aA}^{i+}\land\t_{aA}^{(i+1)-}\land e^{VK^{\a}}\right]$, there exists a constant $\hat C_0>0$ such that
\begin{align}
\P[\hat Y_{n+1}=-k-1|\hat Y_n=-k]\leq\frac12-\hat C_0x^{2i+2}\e^2\equiv\hat p_\S.
\end{align}
\end{lemma}
The proof is a re-run of Step 3 by using the rates \eqref{raten1}.
The rest of the proof of Proposition \ref{staysum} is similar to Step 3.
\end{proof}

%---------------------------
\paragraph{Final Step:}\\[6pt]
\paragraph{Calculation of the Decay Time of $n_{aA}$:}
The following proves Theorem \ref{main1} (ii). Set 
\begin{align}
\sigma\equiv\frac{-\ln(\e K^{1/4-\a})}{\ln(x)}.
\end{align}
Observe that $x^\sigma\e\geq K^{-1/4+\a}$ is just the value until which we can control the decay of $n_{aA}(t)$. Thus, to calculate the time of the controlled decay of the $aA$-population we iterate the system, described above, until $i=\sigma$.
Observe that $\sum_{j=0}^{\s-1}\frac{C_u}{x^j\e}=\frac{C_ux}{1-x}(K^{1/4-\a}-\e^{-1})\geq\frac{C_lx}{1-x}(K^{1/4-\a}-\e^{-1})=\sum_{j=0}^{\s-1}\frac{C_l}{x^j\e}$ and that
the $\theta_j(aA)=\tau_{aA}^{(j+1)-}-\tau_{aA}^{j+}$ are independent random variables.
Thus, for some  constant $\tilde M>M>0$, we get
\begin{align}
\label{time}
&\P\left[\frac{C_ux}{1-x}\left(K^{1/4-\a}-\e^{-1}\right)t\geq\t_{aA}^{\sigma-}-\t_{aA}^{0+}\geq\frac{C_lx}{1-x}\left(K^{1/4-\a}-\e^{-1}\right)\right]\\
	&=\P\left[\frac{C_ux}{1-x}\left(K^{1/4-\a}-\e^{-1}\right)\geq\sum_{j=0}^{\sigma-1}\t_{aA}^{(j+1)-}-\t_{aA}^{j+}\geq\frac{C_lx}{1-x}\left(K^{1/4-\a}-\e^{-1}\right)\right]\\
	&\geq\P\left[\frac{C_u}{x\e}\geq\theta_1(aA)\geq\frac{C_l}{x\e},...,\frac{C_u}{x^\s\e}\geq\theta_\s(aA)\geq\frac{C_l}{x^\s\e}\right]\\
	&\geq1-\s\exp(-MK^{1/2+2\a})
	\geq1-\exp(-\tilde MK^{1/2+2\a}).
\end{align}
\begin{remark}
A function $f(t)=\frac{1}{t}$ needs a time $t_1-t_0=K^{1/4-\a}-\frac{1}{\e}$ to decrease from $f(t_0)=\e$ to $f(t_1)=K^{-1/4+\a}$. Compared to the decay-time of $n_{aA}$ we see that it is of the same order. Thus the stochastic process behaves as the dynamical system.
\end{remark}
\paragraph{Survival until Mutation:}
Now we prove Theorem \ref{main2}.We already know that there is no mutation before a time of order $\frac1{K\mu_K}$ (cf. Lemma \ref{lemma2a} and\eqref{muttime}). Since we have seen that the duration of the first step is $\OO(\ln K)$ and the time 
needed for the second step is bounded, the left hand side of \eqref{mu} ensures that the first two phases are ended before the occurrence of a new mutation. Thus we get the first statement of \eqref{result}. \\
To justify the second statement of \eqref{result}, we have to calculate an upper bound on the mutation time such that there are still enough $aA$-individuals in the population, when the next mutation to a new allele occurs. We saw that we can control the process only until the $aA$-population has decreased to $K^{-1/4+\a}$. Thus we have to verify that the mutation time is smaller than $\OO\left(K^{1/4-\a}\right)$, the time the process $n_{aA}$ needs to decrease to $K^{-1/4+\a}$.\\
The mutation rate of the whole population is the sum of the mutation rates of each subpopulation.
For $t\in[\t_{aA}^{0+},\t_{aA}^{\sigma-}]$ with high probability, the new mutation occurs in the $AA$-population, because $n_{aa}(t)$ and $n_{aA}(t)$ are very small.\\
Let 
\begin{equation}
	p_A(t)=\frac{n_{AA}(t)+\frac{1}{2}n_{aA}(t)}{n_{aa}(t)+n_{aA}(t)+n_{AA}(t)}
\end{equation}
be the relative frequency of $A$-alleles in the population at time $t$. The mutation rate of the $AA$-population is given by
\begin{equation}
\label{mut}
	\mu_{K}^{AA}=\mu_Kfp_A(t)n_{AA}(t) K.
\end{equation}
For $t\in[\t_{aA}^{0+},\t_{aA}^{\sigma-}]$ we know from the results before that $n_{AA}$ is in an $\e$-neighbourhood of its 
equilibrium, $\bar n_A$, and $n_{aA},n_{aa}\leq\e$. We can estimate $\mu_K^{AA}$: 
\begin{align}
\mu_{K}^{AA}=\mu_Kf\bar n_A K+\mathcal{O}(\e).
\end{align}
From \eqref{mut} we get that the time of a new mutation is smaller or equal to $\frac{1}{f\bar n_AK\mu_K}+\OO(\e)$. Thus to ensure that we still have $a$-alleles in the population, we have to choose $\mu_K$ in such a way that 
\begin{align}
\frac{1}{f\bar n_aK\mu_K}&\ll \frac{C_lx}{1-x}K^{1/4-\a},		
\end{align}
since we can ensure the survival of $n_{aA}(t)$ until the time $\frac{C_lx}{1-x}K^{1/4-\a}$ (cf. \eqref{time}). Hence, the right hand side of \eqref{mu} gives us that a new mutation occurs before the $aA$-population died out. This finishes the proof of Theorem \ref{main2}.

%-----------------------BIB----------------------------------------------------
%\bibliographystyle{abbrv}
%\bibliographystyle{apalike}
%\bibliographystyle{apanew}

\end{document}